\documentclass[12pt,english]{amsart}
\usepackage[utf8]{inputenc}
\usepackage[T1]{fontenc}
\usepackage[english, french]{babel}
\usepackage{amsmath}
\usepackage{amssymb}
\usepackage{url}
\usepackage{xspace}
\usepackage{amsthm}
\usepackage{mathrsfs}
\usepackage{newtxtext}
\usepackage{newtxmath}
\usepackage[top=1.15in,bottom=1.15in,left=1.15in,right=1.15in,marginpar=1in]{geometry}
\usepackage{verbatim}
\usepackage{caption}
\usepackage{subcaption}
\usepackage{graphicx}
\usepackage{pgfplots}
\usepackage{amsfonts}
\usepackage{mathtools}
\usepackage{tikz-cd}
\usepackage{hyperref}
\usepackage{ifthen}
\usepackage{stackengine}
\usepackage{xcolor}
\usepackage{tabularx}
\usepackage{tikz}
\usepackage{comment}
\usepackage{bigints}
\stackMath

\newcommand{\D}{\mathbb{D}}
\newcommand{\C}{\mathbb{C}}
\newcommand{\R}{\mathbb{R}}

\newcommand{\Z}{\mathbb{Z}}

\newtheorem{thm}{Theorem}[section]
\newtheorem{defn}[thm]{Definition}
\newtheorem{lemma}[thm]{Lemma}
\newtheorem{rem}[thm]{Remark}

\newtheorem{prop}[thm]{Proposition}
\newtheorem{conj}[thm]{Conjecture}

\title[Embedding Unicritical Connectedness Loci]{Embedding Unicritical Connectedness Loci}

\author[Malavika Mukundan]{Malavika Mukundan}
\address{University of Michigan, Ann Arbor}
\email{malavim@umich.edu} 
\urladdr{http://www-personal.umich.edu/$\sim$ malavim/index.html}
\begin{document}
\selectlanguage{english} 
\begin{abstract}    In this article, for  degree $d\geq 1$, we construct an embedding $\Phi_d $ of the connectedness locus $\mathcal{M}_{d+1}$ of the polynomials $z^{d+1}+c$ into the connectedness locus of degree $2d+1$ bicritical odd polynomials.
\end{abstract}
\maketitle

\section{Introduction}
Relationships between different families of rational maps have been studied in various contexts in complex dynamics. In rational dynamics, quadratic polynomials of the form $z^2+c$ are the fundamental objects of study, and much of the field involves the study of the Mandelbrot set pioneered by Douady and Hubbard, in  \cite{Hubbarditer}, \cite{MR728980},\cite{MR812271} and \cite{MR1215974}, and developed by Milnor (\cite{MR1755445}), Lyubich and Dudko (\cite{https://doi.org/10.48550/arxiv.1808.10425}) and several others. In general, polynomials with a single critical point, normalized as $z^d+c$, and their connectedness loci $\mathcal{M}_d$ - that is, the set of parameters $c$ for which the filled Julia set is connected, commonly referred to as the Multibrot sets, have also been studied in \cite{Schleicher98onfibers}, \cite{MR3444240}, etc., and the properties of these sets have been used to conjecture and prove several results in both rational and transcendental dynamics.

    For example, the theory of matings and the work of Tan Lei, Rees, Shishikura and others (see \cite{Rees1990}, \cite{Shishikura2000OnAT}, \cite{lei_1992}) created a link between polynomials and rational maps, by combining two polynomials to create a rational map. This made it easier to study certain hyperbolic components in rational parameter spaces, as well as the structure of Julia sets of rational maps that arise as matings.
    
    \begin{figure}
    \centering
     \begin{subfigure}[b]{0.4\textwidth}
         \includegraphics[width=\textwidth]{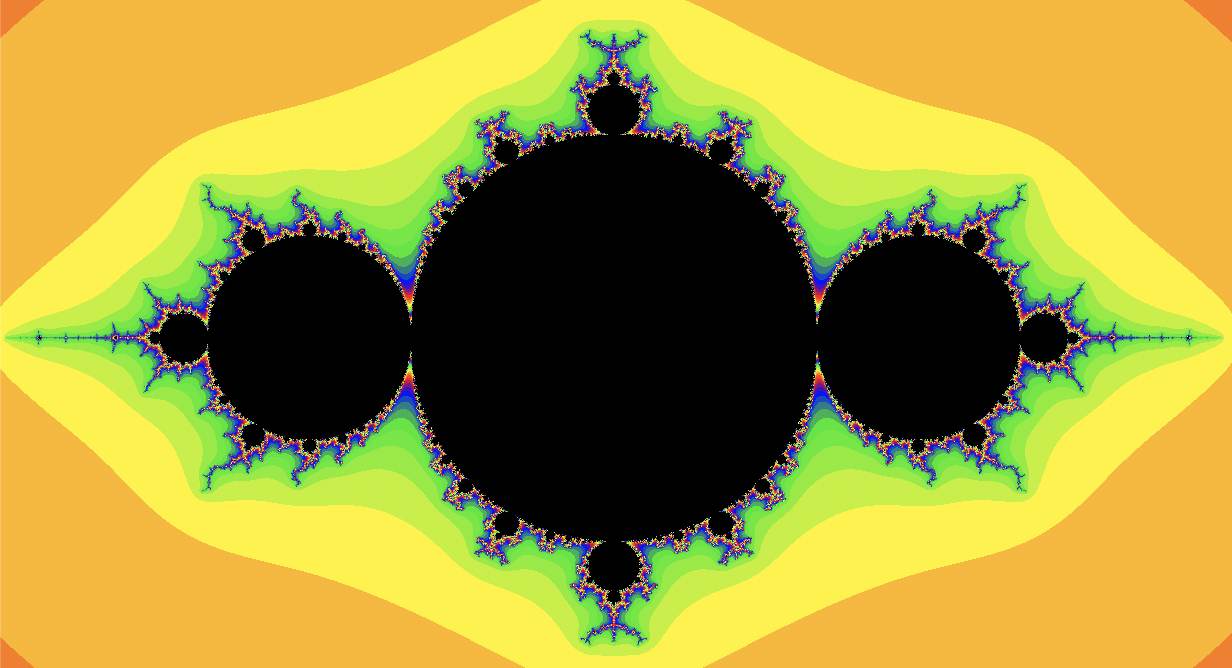}
         \caption{$\mathcal{CBO}_1$}
         \label{fig:odd1}
     \end{subfigure}
     \hfill
     \begin{subfigure}[b]{0.4\textwidth}
         \includegraphics[width=\textwidth]{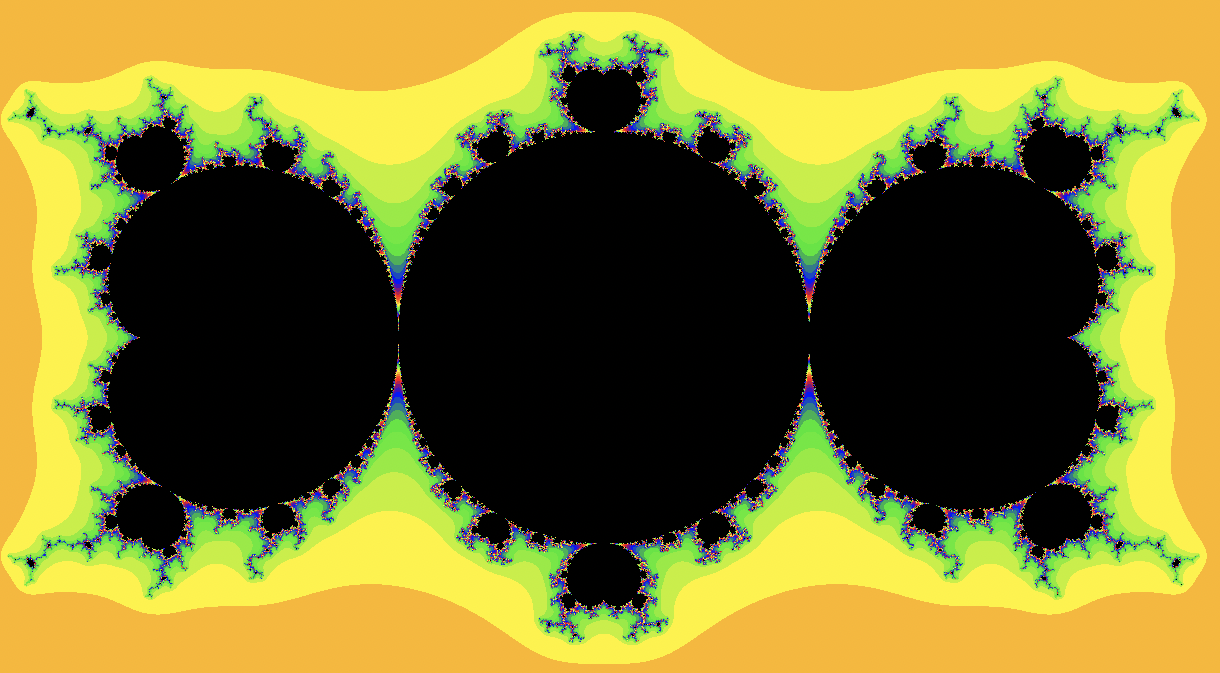}
         \caption{$\mathcal{CBO}_2$}
         \label{fig:odd2_1}
     \end{subfigure}\\     
     \begin{subfigure}[b]{0.4\textwidth}
         \includegraphics[width=\textwidth]{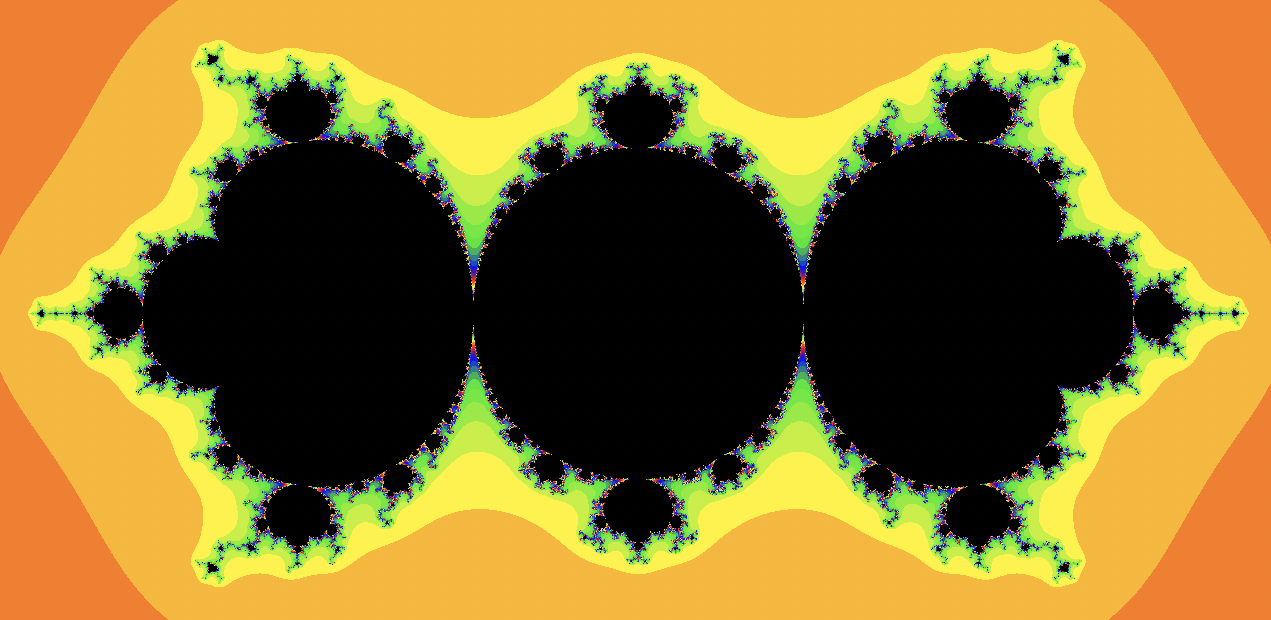}
         \caption{$\mathcal{CBO}_3$}
         \label{fig:odd3}
     \end{subfigure}\hfill
\begin{subfigure}[b]{0.4\textwidth}
         \includegraphics[width=\textwidth]{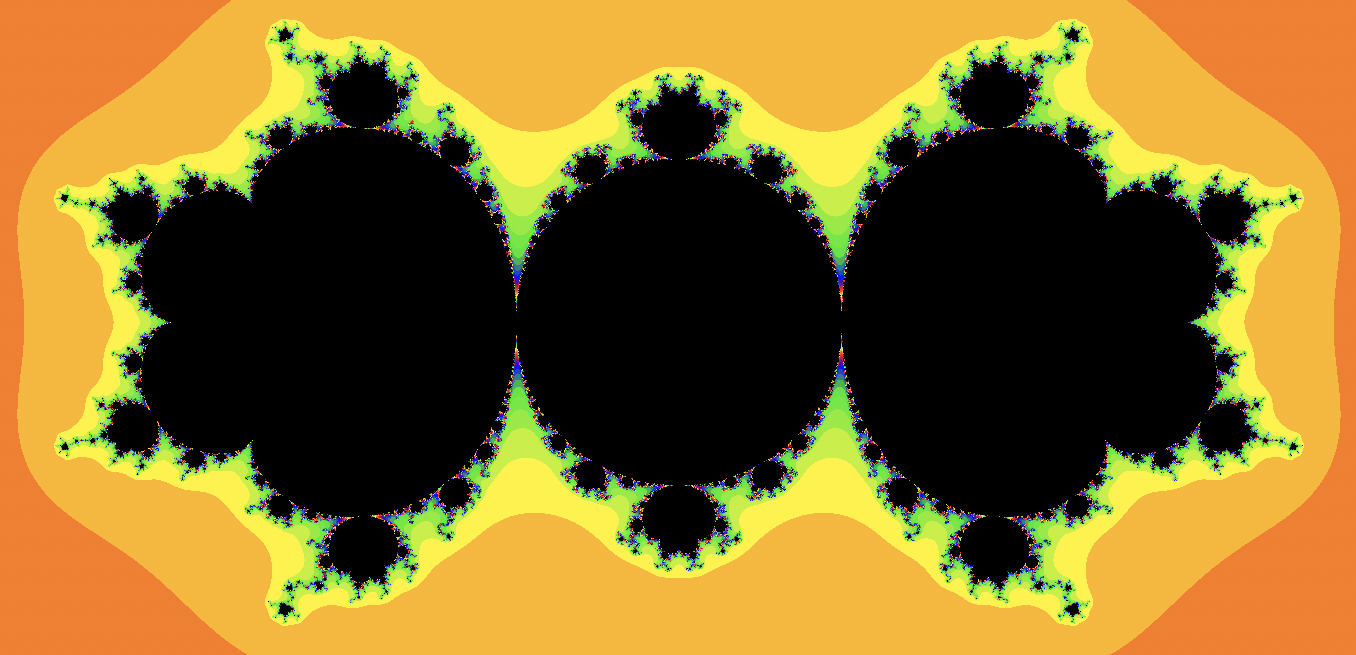}
         \caption{$\mathcal{CBO}_4$}
         \label{fig:odd4}
     \end{subfigure}\\     \begin{subfigure}[b]{0.4\textwidth}
         \includegraphics[width=\textwidth]{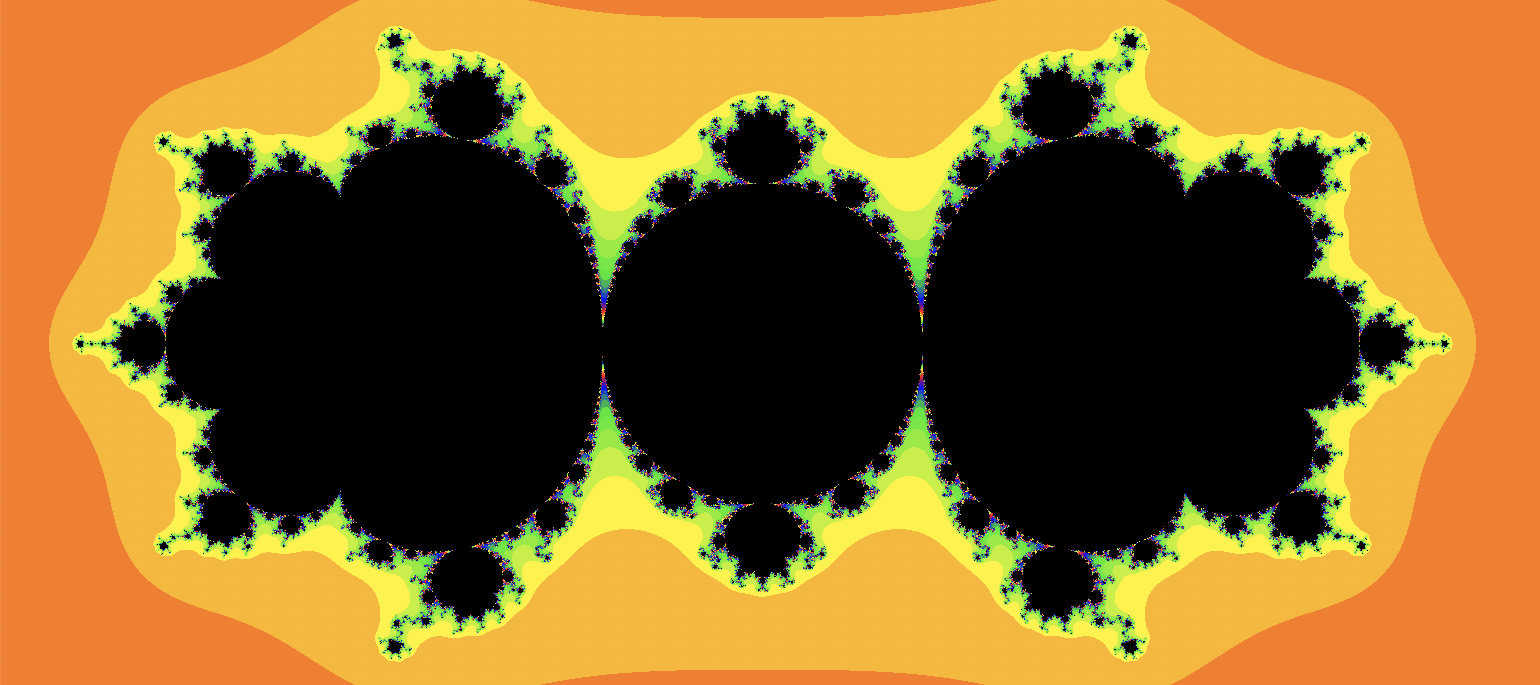}
         \caption{$\mathcal{CBO}_5$}
         \label{fig:odd5}
     \end{subfigure} 
\caption{The families $\mathcal{CBO}_d$ for $d=1,2,3,4,5$}
     \label{fig:odds}
\end{figure}

    Renormalization is a powerful tool in dynamics. It restricts certain holomorphic functions to smaller domains in which they ``look'' like some $z^d+c$ to make them easier to study. In his study of renormalizable maps in \cite{McMullen19972TM}, McMullen proved that unicriticals are universal in the sense that there are small copies of the Multibrot sets found in any holomorphic family of maps. 
    
    Branner and Douady constructed a continuous map from the basilica limb of the Mandelbrot set $\mathcal{M}_2$ to the rabbit limb (see \cite{10.1007/BFb0081395}). This was later extended by Branner and Fagella in \cite{MR1757453} into homeomorphisms between various limbs of the Mandelbrot set, and in a different spirit, by Dudko and Schleicher in \cite{MR2888182}. In \cite{riedl}, Riedl and Schleicher also construct a homeomorphism from a subset of any $\frac{p}{nq}$-limb of the Mandelbrot set to the $\frac{p}{q}$-limb. 
    
    We establish another such relationship between two holomorphic families -  unicritical polynomials and the family of symmetric polynomials which we introduce below and describe in detail in Section~\ref{section:prelim}. We take symmetric polynomials to mean polynomials that commute with some affine map $M$ satisfying $M^{\circ 2} = \text{Id}$. Symmetric cubic polynomials are encountered, for example, in the study of core entropy (see  \cite{gao2019core}). We focus on the more specific family of symmetric polynomials with exactly two critical points. As we will show in Section~\ref{section:prelim}, such a polynomial is affine conjugate to  
    $$p_{a,d}(z) = a\int_0^z\Big(1-\frac{w^2}{d}\Big)^ddw$$
for some $a \in \C^*$ and $d\geq 2$. For each $a$, 
    $p_{a,d}$ is an odd polynomial- that is, it commutes with $z \mapsto -z$. For a fixed $d$, we let $p_a = p_{a,d}$. We let $\mathcal{CBO}_d$ denote the set of $a \in \C^*$ such that $p_a$ has connected filled Julia set. This is a closed, compact connected subset of $\C$. In this article we shall prove the following.
    \begin{thm}\label{thm:maintheorem}
    For $d\geq 1$, there exists a continuous map $\Phi_d: \mathcal{M}_{d+1} \longrightarrow \mathcal{CBO}_d$ that is a homeomorphism onto its image.
    \end{thm}Our proof is along the lines of Douady and Branner's use of quasiconformal surgery in \cite{10.1007/BFb0081395}. We shall perform a quasiconformal surgery along a $\beta-$ fixed point and its pre-images. The map $\Phi_d$ is natural in the sense that its inverse can be described by a renormalization operator on a subset of $\mathcal{CBO}_d$. We shall also give a complete description of the image under $\Phi_d$ (see Section~\ref{section:imgdefn}).
    
    Although we do not provide details here, our construction holds in the following generality:
    \begin{thm}
    For any integer $k\geq 2$, there exists a continuous map from $\mathcal{M}_{d+1}$ to the collection of $a \in \C^*$ such that $a {\displaystyle \int_{0}^{z} } \Big(1-\frac{w^k}{d}\Big)^ddw$ has connected Julia set, that is a homeomorphism onto its image. 
    \end{thm}The family $p_{a,d}$ is interesting in its own right: as $d\longrightarrow \infty$, $p_{a,d}(z) \longrightarrow a{\displaystyle \int_{0}^{z}} e^{-w^2}dw$ locally uniformly on $\C$. The limit function is entire, odd, has two asymptotic values $\pm \frac{a\sqrt{\pi}}{2}$ and no critical points. It is called an ``error'' function (see \cite{nevanlinna1970analytic} for an introduction). Error functions belong to the larger Speiser class- the family of entire functions with finitely many critical and asymptotic values. This family is studied in \cite{goldberg_keen_1986}, \cite{AIF_1992__42_4_989_0} and several others. 
    
    The simplest of the Speiser class is the family of exponential functions. A lot of the analysis of exponential functions is a direct application of the tools used in the analysis of unicritical polynomials, normalized as $\lambda(1+\frac{z}{d})^d$ and using the fact that they converge to $\lambda \exp{z}$ as $d \longrightarrow \infty$. This is a theme that is explored in \cite{MR1785056}. Our work in progress aims to carry out a similar analysis for the error functions $a{\displaystyle \int_0^z}e^{-w^2}dw$. This paper presents a structural similarity between the  polynomials approximating exponential functions, and the polynomials $p_{a,d}$ that approximate error functions, and prompts us to make the following conjecture:
    \begin{conj}
    Let $E_c(z) = \exp z+c$, and $\mathscr{E}_a(z)  = a {\displaystyle \int_{0}^{z}} e^{-w^2}dw$.
    There exists a continuous map from $\big\{c \in \C| \{E_c^{\circ n}(c)\}_{n \geq 0} \text{ is bounded}\big\}$ to the set $\big\{a \in \C^*| \{\mathscr{E}^{\circ n}_a(a)\}_{n\geq 0} \text{ is bounded}\big\}$ that is a  homeomorphism onto its image.
    \end{conj}
    There is some evidence to show that this is reasonable; work in progress indicates that it may be possible to embed postsingularly finite exponential functions into the collection of postsingularly finite error functions in a combinatorially meaningful manner. We do not, however, address error functions in this article.
    
    The paper is organized as follows. In Section~\ref{section:prelim}, we introduce symmetric polynomials and establish some of their basic properties, provide motivation for Theorem~\ref{thm:maintheorem} while laying out our proof strategy, and describe the image of $\Phi_d$. In Sections~\ref{section:defn} and ~\ref{section:continuity} respectively, we define $\Phi_d$ and prove that it is continuous. We end in Section~\ref{section:injectivity} by constructing a continuous inverse for $\Phi_d$ on its image.
 \subsection*{Acknowledgements.}The author is indebted to John Hubbard and Sarah Koch for their continued guidance throughout this study, and to Dierk Schleicher for helping widen the horizons of the author's research perspective. Special thanks to Jack Burkart, Alex Kapiamba, Leticia Pardo-Sim\'{o}n and Vasiliki Evdoridou for helpful conversations and resources, and to Lukas Geyer, N\'{u}ria Fagella, Lasse Rempe, Laurent Bartholdi, Joanna Furno, Giulio Tiozzo, Eriko Hironaka, David Mart\'{i}-Pete, Mikhail Hlushchanka, Nikolai Prochorov and others for their time and insightful comments. The author also thanks  Daniel Stoll for an introduction to  the Mandel software package.

This material is based upon work supported by the National Science Foundation under Grant No.\ DMS-1928930
while the author was in residence at the Mathematical Sciences Research Institute in Berkeley, California, during
the Spring 2022 semester.
\begin{figure}
    \centering
     \begin{subfigure}[b]{0.4\textwidth}
         \centering
         \includegraphics[width=\textwidth]{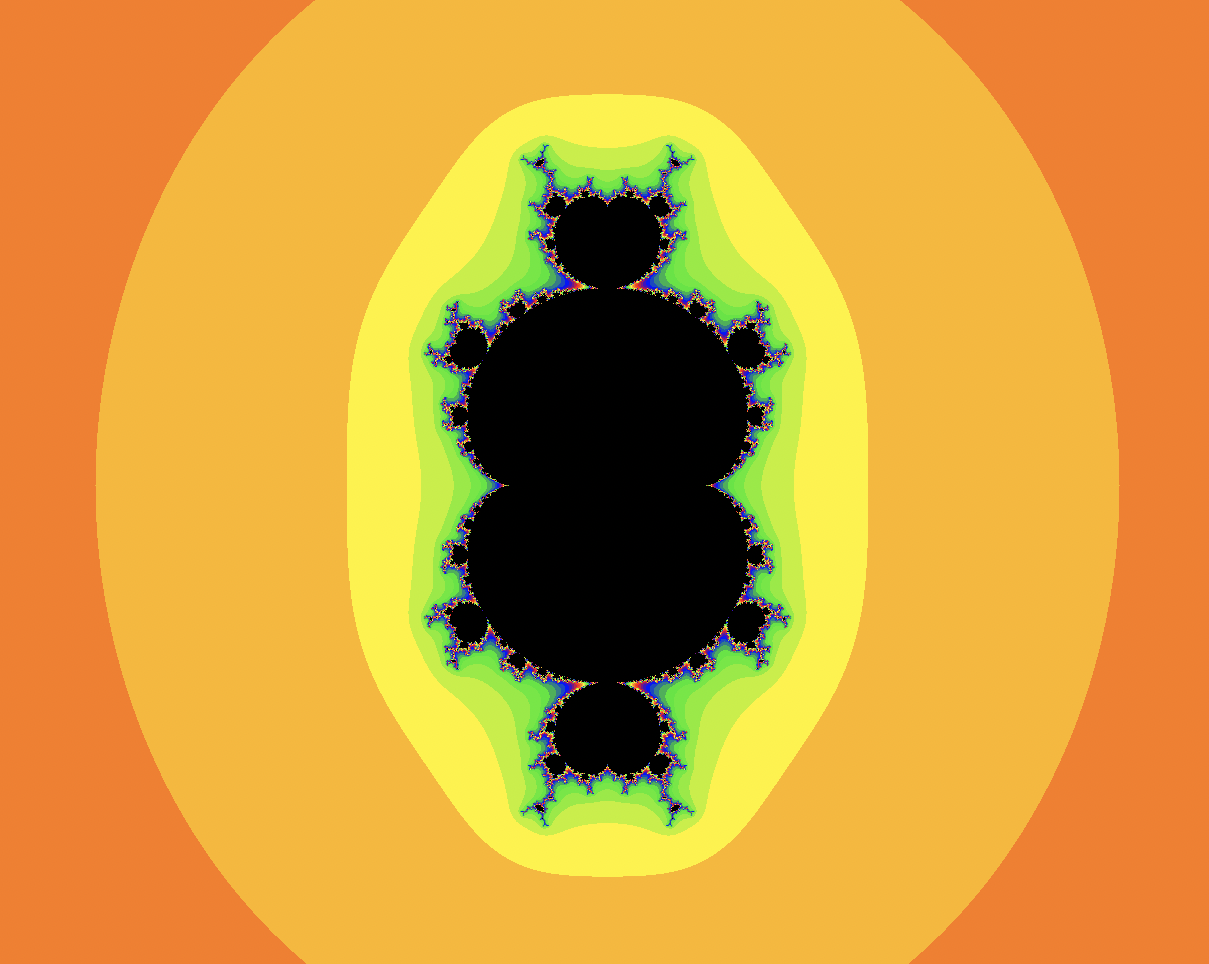}
         \caption{$\mathcal{M}_3$}
         \label{fig:m3}
     \end{subfigure}
     \hspace{5pt}
     \begin{subfigure}[b]{0.4\textwidth}
         \centering
         \includegraphics[width=\textwidth]{Images/odd2.png}
         \caption{$\mathcal{CBO}_2$}
         \label{fig:odd2_2}
     \end{subfigure}\\ 
     \begin{subfigure}[b]{0.5\textwidth}
         \centering
         \includegraphics[width=\textwidth]{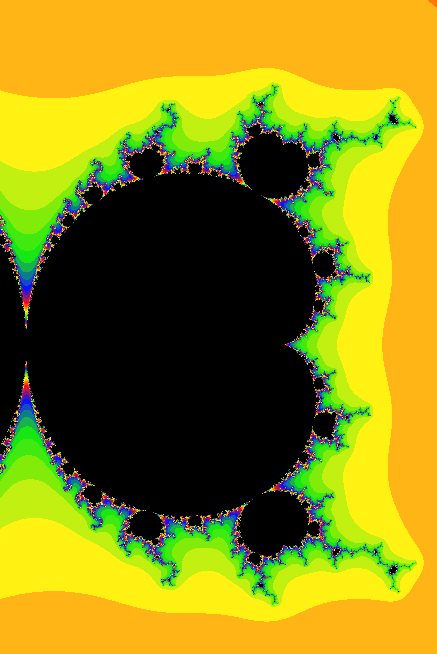}
         \caption{A portion of $\mathcal{CBO}_2$. The cut point on the mid-left is $a=1$. Note the resemblance to $\mathcal{M}_3$}
         \label{fig:odd2_right}
     \end{subfigure}
     \label{fig:fig2}
     \caption{The figure on the top left is the unicritical locus $\mathcal{M}_3$. The figure on the top right is the locus $\mathcal{CBO}_2$ of odd bicritical polynomials of degree $5$. The figure at the bottom zooms in on the right of the figure on the top right- we will show that this region contains a copy of $\mathcal{M}_3$ }
\end{figure}
\section{Preliminaries}\label{section:prelim}
For an introduction to polynomial dynamics and the Multibrot sets, see \cite{MR1755445}, \cite{MR2193309}, \cite{MR3675959} and \cite{MR3444240}.
\subsection{Introduction}
We call a polynomial  $f$ of degree $>1$ symmetric if it commutes with an affine map $M$ that satisfies $M^{\circ 2}(z) = z$. $M$ is of the form $M(z) = -z+b$ for some $b \in \C$, and we may conjugate $M$ by a translation $\tau$ so that $\tau^{-1}\circ M \circ \tau(z) = -z$. We have
\begin{align*}
   \tau^{-1}\circ f \circ \tau(z) & = \tau^{-1}\circ (M \circ f \circ M^{-1} )\circ \tau(z) \\& = (\tau^{-1}\circ M \circ \tau) \circ (\tau^{-1} \circ f \circ \tau) \circ (\tau^{-1} \circ M^{-1}\circ \tau)(z) \\& = -(\tau^{-1} \circ f \circ \tau)(-z)
\end{align*}
That is, $\tau^{-1} \circ f \circ \tau$ is an odd polynomial. Therefore, every symmetric polynomial contains an odd polynomial in its affine conjugacy class. We recall that an odd polynomial has only odd degree terms.
\subsection{Bicritical odd polynomials} $f$ is bicritical if it has, upto multiplicity, exactly two critical points on the plane. Let $f$ be a bicritical odd polynomial of degree $2d+1$, with $d\geq 1$. The Riemann Hurwitz formula shows that $f$ has local degree $d+1$ at both critical points. Furthermore, the critical points are of the form $\pm x$ for some $x \in \C^*$. 
Let $\phi(z) = kz$ be such that $\phi(\{x,-x\}) = \{\sqrt{d},-\sqrt{d}\}$. Then there exists a constant $a \in \C^*$ such that
\begin{align*}
(\phi \circ f \circ \phi^{-1})'(z) & = a\Big(1-\frac{z^2}{d}\Big)^d
\end{align*}
Therefore, 
\begin{align*}
    \phi \circ f \circ \phi^{-1}(z) & = a\int_0^z\Big(1-\frac{w^2}{d}\Big)^ddw
\end{align*}
\noindent For $a \in \C^*$, let 
\begin{align*}
    p_a(z) & = a\int_0^z \Big(1-\frac{w^2}{d}\Big)^ddw
\end{align*}

It is evident that $p_{a}$ is affine conjugate to $p_{a'}$ if and only if $a = a'$. Therefore, the space of bicritical odd polynomials modulo conjugation by scaling (or, the space of symmetric bicritical polynomials modulo affine conjugation) is the  family $a \mapsto p_a$ over $\C^*$. We shall denote this family $\mathcal{BO}_d$, and let 
\begin{align*}
    \mathcal{CBO}_d = \{a: K_{p_a} \text{ is connected}\}
\end{align*}
Figure~\ref{fig:odds} illustrates $\mathcal{CBO}_d$ for $d=1,2,3,4,5,19$.

We consider the part of $\mathcal{CBO}_2$ illustrated in Figure~\ref{fig:odd2_right}, and present some of the Hubbard trees of postcritically finite polynomials in this region, in Figure~\ref{fig:trees}.
The pictures indicate a relationship between $\mathcal{M}_{3}$ and $\mathcal{CBO}_2$, and in general, between $\mathcal{M}_{d+1}$ and $\mathcal{CBO}_d$. 
\begin{figure}
\centering
    \begin{subfigure}[b]{0.5\textwidth}
    \centering
         \includegraphics[width=\textwidth]{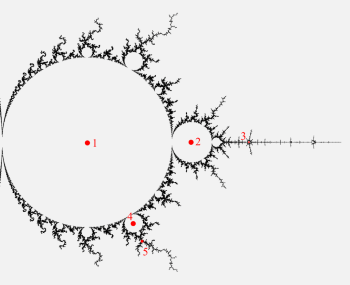}
         \caption{A section of $\mathcal{CBO}_1$}
         \label{fig:odd1_right}
    \end{subfigure}\\
    \begin{subfigure}[b]{0.7\textwidth}
    \centering
         \includegraphics[width=\textwidth]{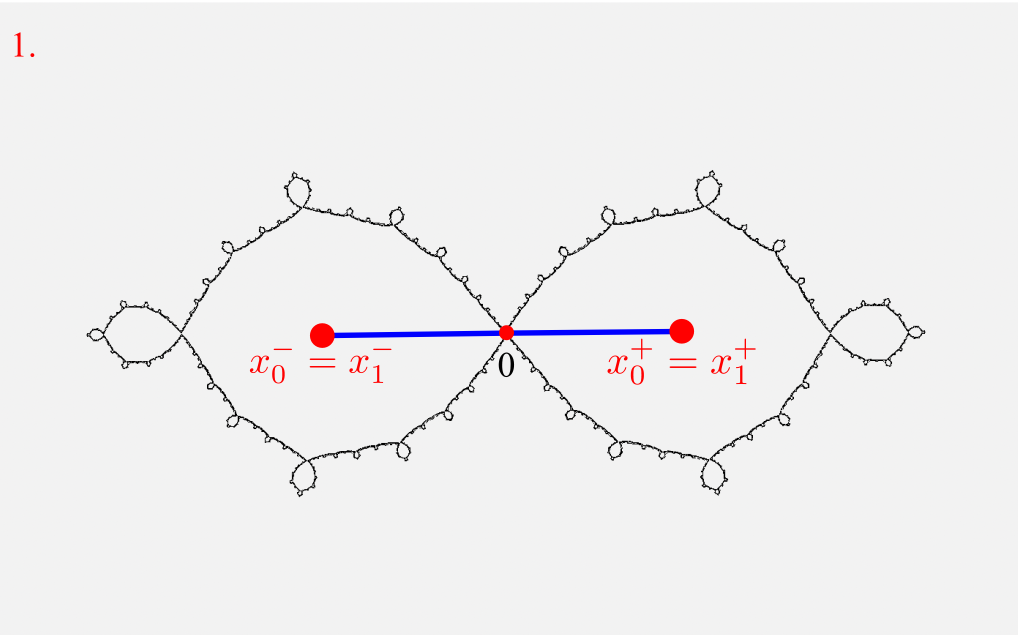}
         \caption{The $(+,-)$ type $z\mapsto z^2$}
         \label{fig:maincardioid}
    \end{subfigure}\\
     
    \begin{subfigure}[b]{0.705\textwidth}
    \centering
         \includegraphics[width=\textwidth]{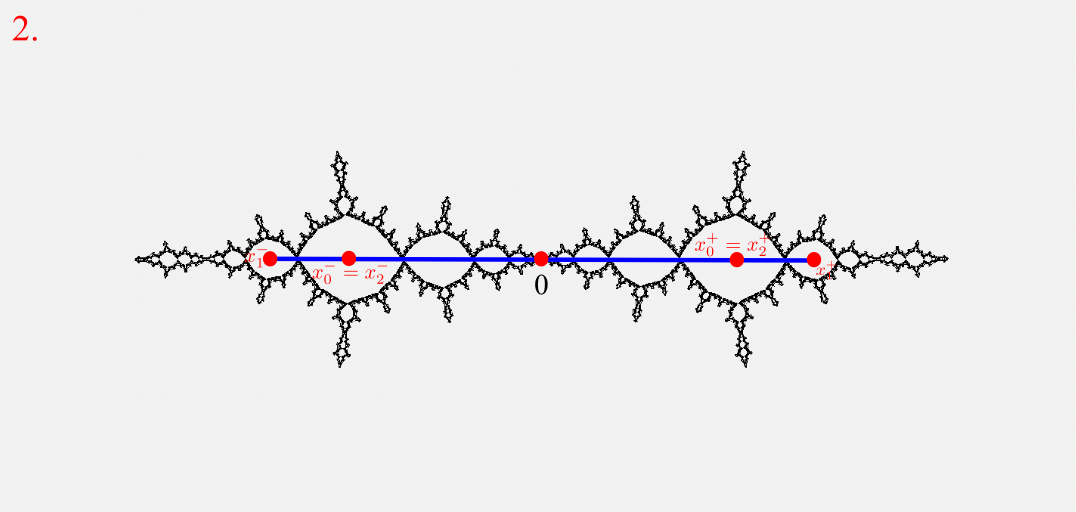}
         \caption{The $(+,-)$ type basilica}
         \label{fig:basilica}
     \end{subfigure}\hfill
    \centering
    \caption{Postcritically finite polynomials along with their Hubbard trees in the family $p_a(z)=a{\displaystyle \int_0^z}(1-w^2)dw$. $x_0^- = -1$, $x_0^+=1$ are the two critical points, with $x_i^{\pm} = p_a^{\circ i}(x_0^\pm)$. Terminology:`$(+,-)$ type polynomial $p$' refers to the polynomial in $\mathcal{CBO}_1$ that looks like a pair of copies of the polynomial $p = f_c$, where $c \in \mathcal{M}_2$}
    \label{fig:trees}
\end{figure}
\begin{figure} 
\centering
\begin{subfigure}[b]{0.9\textwidth}
\centering
         \includegraphics[width=\textwidth]{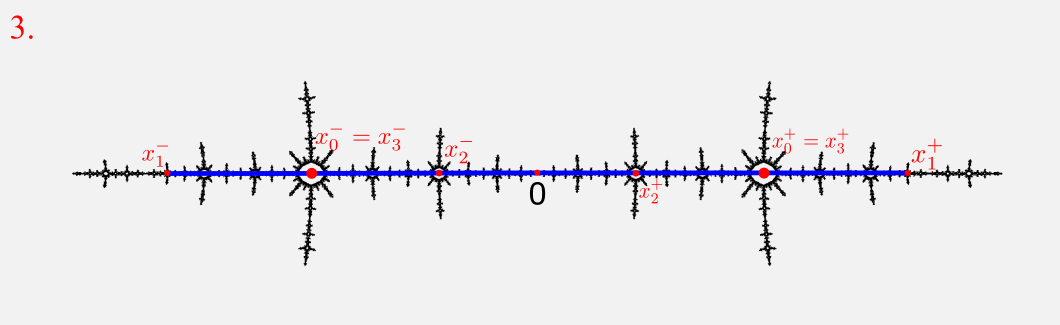}
         \caption{The $(+,-)$ type airplane}
         \label{fig:airplane}
     \end{subfigure}\\
     \begin{subfigure}[b]{0.9\textwidth}
     \centering
         \includegraphics[width=\textwidth]{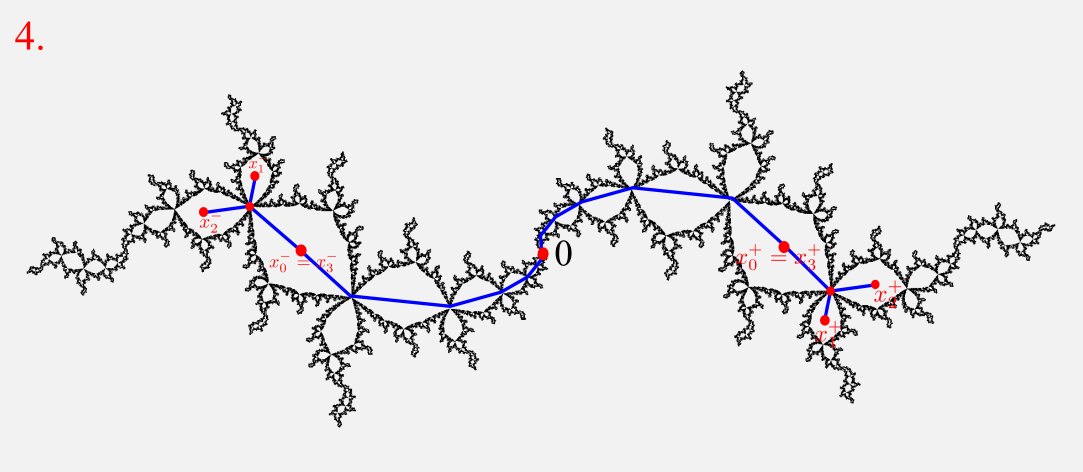}
         \caption{The $(+,-)$ type rabbit}
         \label{fig:rabbit}
     \end{subfigure}\\
     \begin{subfigure}[b]{0.9\textwidth}
     \centering
         \includegraphics[width=\textwidth]{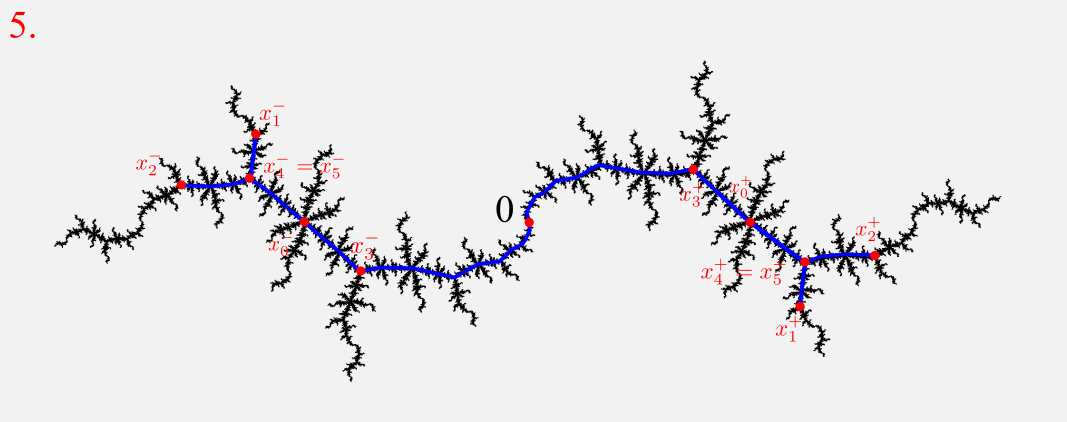}
         \caption{The $(+,-)$ type $z\mapsto z^2 -0.10003 + 0.95227i $}
         \label{fig:trip_point}
     \end{subfigure}
     \centering
    \caption{More examples of Hubbard trees in $\mathcal{CBO}_1$}
    \label{fig:trees2}
\end{figure}
\subsection{Monic representatives of polynomials in $\mathcal{BO}_d$}\label{section:imgdefn}
Any polynomial $p_a$ for $a \in \mathcal{BO}_d$ has leading coefficient $T(a):=\frac{(-1)^da}{d^d(2d+1)}$ attached to $z^{2d+1}$. Let $w(z) = \frac{z}{s}$. We note that $P_s(z)=w^{-1} \circ p_a \circ w$ is monic if and only if $s^{2d} = T(a)$.

The polynomial $P_s$ admits a unique B\"{o}ttcher chart $\varphi_s$ in a neighborhood of $\infty$ that satisfies $\lim_{z \rightarrow \infty}\frac{\varphi_s(z)}{z} = 1$, and if $s_1^{2d} = s_2^{2d}$, then $\varphi_{s_1}(z) = \omega \varphi_{s_2}(z)$ where $\omega = \frac{s_2}{s_1}$ is a $2d-$th root of unity. Let $\mathcal{R}_\theta(s)$ denote the ray at angle $\theta$ in the dynamical plane of $P_s$. Then it is easy to see that if $\frac{s_2}{s_1} = e^{\frac{2\pi ij}{2d}}$ for some integer $j$, then for all $\theta \in \R/\Z$, 
\begin{align*}
    \mathcal{R}_\theta(s_2) & =  \mathcal{R}_{\theta+\frac{j}{2d}}(s_1)
\end{align*}
Additionally, since $P_s$ is odd, $\varphi_s(z) = \lim _{n \rightarrow \infty} P_s^{\circ n}(z) ^{\frac{1}{(2d+1)^n}}$ satisfies 
\begin{align*}
    \varphi_s(-z) & = -\varphi_s(z)\\
    \implies \mathcal{R}_{\theta+\frac{1}{2}}(s) & = -\mathcal{R}_\theta(s)
\end{align*}

For any $s$ such that $0 \in J_{P_s}$, there exists a subset $\Theta$ of $\big\{0,1,...,\frac{2d-1}{2d}\big\}$ satisfying $\Theta+\frac{1}{2}=\Theta$ such that the dynamical rays landing at $0$ are exactly those with angles in $\Theta$. Moreover, if $s'=e^{\frac{2 \pi i j}{2d}} s$, then the set of angles that land at $0$ in the dynamical plane of $P_{s'}$ is $\frac{j}{2d}+\Theta$.

This shows that for any $a \in \mathcal{CBO}_d$ such that $0 \in J_{p_a}$, there exists a monic representative  $P_s$ of $p_a$ so that $0$ is the landing point of $\mathcal{R}_0(s)$ and $\mathcal{R}_\frac{1}{2}(s)$. The union of the rays  $\mathcal{R}_0(s)$ and $\mathcal{R}_\frac{1}{2}(s)$  separates the plane into two connected components $F^L_s$ and $F^R_{s}$, named so that $F^L_{s}$ is the component that contains the critical point $-s\sqrt{d}$, and $F^R_s$ contains $s\sqrt{d}$ ($L$ and $R$ stand for left and right).
\subsubsection{The Image under $\Phi_d$}
The set of $a \in \mathcal{CBO}_d$ such that $0 \in J_{p_a}$ is exactly $\mathcal{CBO}_d \setminus \D$. Let $H$ be the component of $\mathcal{CBO}_d \setminus \D$ that intersects the right half plane. By the previous paragraph, on $H$, there exists a branch of $a \mapsto (T(a))^\frac{1}{2d}$, which we shall denote $s(a)$, so that $0$ is the landing point of $\mathcal{R}_0(s(a))$ and $\mathcal{R}_\frac{1}{2}(s(a))$. 

We note that the critical points of $P_{s}$ are $\pm s \sqrt{d}$. Let $\mathscr{O}^R_s$ and $\mathscr{O}^L_s$ be the orbits under $P_s$ of $s\sqrt{d}$ and $-s\sqrt{d}$ respectively.
\begin{align*}
    \Phi_d(\mathcal{M}_{d+1})& = \big\{a\big| \mathscr{O}^L_{s(a)} \subset F^L_{s(a)} \cup \{0\} \text{ and }\mathscr{O}_{s(a)}^R \subset F^R_{s(a)} \cup \{0\}\big \}
\end{align*}
That is, the image is the set of polynomials where the dynamical rays at angles $0,\frac{1}{2}$ separate the orbits of the two distinct critical points. It is easy to see that the latter is a closed set in $\mathcal{CBO}_d$. We shall henceforth denote this image as $\mathcal{CBO}^{(+,-)}_d$. This is also a proper subset of $\mathcal{CBO}_d$: for example, $a=-1$ is not in this set. We have described in detail the dynamics of polynomials in $\mathcal{CBO}_d^{(+,-)}$ in  Section~\ref{section:imgdynamics}.

Figure~\ref{fig:odd1_right} illustrates the way $\Phi_d$ maps $\mathcal{M}_{2}$ by pointing out the position of the images of well-known polynomials like the rabbit, co-rabbit, airplane, etc.  
\subsection{Quotienting by $z^2$}
Given $p_a \in \mathcal{BO}_d$, there exists a unique polynomial $\mathcal{P}_a:\hat{\C} \longrightarrow \hat{\C}$ so that the following diagram commutes.\\
\[
\begin{tikzcd}
\hat{\C} \arrow[d,"z\mapsto z^2"] \arrow[r,"p_a"] & \hat{\C} \arrow[d,"z\mapsto z^2"]\\
\hat{\C} \arrow[r,"\mathcal{P}_a"] & \hat{\C} 
\end{tikzcd}
\]
The critical points of $\mathcal{P}_a$ are $d$ and $\big\{x_\ell^2\big\}_{\ell=1}^d$, where $\pm x_\ell$, $\ell = 1,2,...,d$, are the pre-images of $0$ that are not equal to $0$. 
When $d=1$, the family $\mathcal{P}_a$ corresponds to the collection of cubic polynomials where one critical point is a pre-image of a $\beta-$fixed point (that is, the landing point of a dynamical ray at angle $0$ or $\frac{1}{2}$) and the other is free. This is isomorphic to the collection $\mathcal{F} = \big\{(a,b) | Q_{a,b}(a) = -2a\big\}$, where $Q_{a,b}(z) = z^3 - 3a^2z+b$ discussed in \cite[Chapters~I,II]{10.1007/BFb0081395} in the following way: letting $\widetilde{a} = 9a^2$, we have
\begin{align*}
    \mathcal{P}_{-a} &=\mathcal{P}_a \simeq Q_{\widetilde{a},2\widetilde{a}^3-2\widetilde{a}}
\end{align*}
where $\simeq$ refers to affine conjugacy.

Let $F_+ \subset \mathcal{F}$ be the collection of polynomials $Q_{a,b}$ for which the critical point $-a$ maps to the landing point of the dynamical ray at angle $0$, and the other critical point $a$ is in the filled Julia set. Douady and Branner  show that there exists a homeomorphism $\Phi_B$ from the basilica limb of the Mandelbrot set to $F_+$. The relationship between  $\mathcal{CBO}_1^{(+,-)}$ and $F_+$ is as follows:
\begin{align*}
     \big\{\big(\widetilde{a},2{\widetilde{a}}^3 - 2\widetilde{a}\big) \big|a \in \mathcal{CBO}_1^{(+,-)}, \widetilde{a} = 9a^2\big\} \subsetneq  F_+
\end{align*}
For $d=1$, the map $\Phi_d$ we construct in this paper exhibits different behaviour from the $\Phi_B$ that the authors construct in \cite[Chapter~II]{10.1007/BFb0081395}. Firstly, it is defined on the whole of the Mandelbrot set, and not just the basilica limb. Secondly, generally, given $c$ in the basilica limb, if $Q_{\widetilde{a},2\widetilde{a}^3 - 2\widetilde{a}}$ is the polynomial corresponding  to $\Phi_B(c)$, $\widetilde{a}$ does not equal $9\Phi_d(c)^2$. Thirdly, it is evident that our map does not change the combinatorics of critical portraits, whereas $\Phi_B$ does.
\subsection{Properties of $\mathcal{CBO}_d$}Let $\mathcal{MBO}_d$ denote the set of $s$ such that $P_s$ has connected filled Julia set. Using the methods used by Douady and Hubbard in their proof that the Mandelbrot set is connected (see, for example, \cite[Chapter~10]{MR3675959}), we find that the map
\begin{align*}
    \hat{\C} \setminus \mathcal{MBO}_d & \longrightarrow \hat{\C} \setminus \overline{\D}\\
    c &\mapsto \varphi_c(P_s(-s\sqrt{d}))\\
    \infty &\mapsto \infty
\end{align*}
is a proper map of degree $2d$ ramified at $\infty$. This implies $\mathcal{MBO}_d$, and consequently $\mathcal{CBO}_d$, is connected and compact. 
\subsection{Proof Strategy and Tools}
We will use all the theorems listed in this section. Their statements are borrowed from \cite{Douady1985}.
\begin{thm}(The Measurable Riemann Mapping Theorem)
Let $\mu_0$ be the standard complex structure on $\C$. If $\mu$ is a complex structure on a simply connected domain $U \subset \C$ that has bounded dilitation with respect to $\mu_0$, then there exists a quasiconformal map $f:U \longrightarrow V\subset \hat{\C}$ satisfying
\begin{align*}
f^*\mu_0 = \mu
\end{align*}
unique up to post composition by a M\"{o}bius transformation.
\begin{enumerate}
    \item Let $\mu_n $ be a sequence of Beltrami forms on a bounded domain $U \subset \C$ such that $||\mu_n||_\infty \leq m<1$ and $\mu_n \longrightarrow \mu$ in the $L^1$ norm, where $\mu$ is a Beltrami form on $U$ with $||\mu||_\infty \leq m$. There exists a sequence of integrating maps $\phi_n$ for $\mu_n$ and an integrating map $\phi$ for $\mu$ such that $\phi_n\longrightarrow \phi$ uniformly on $U$.
    \item 
Let $\Lambda$ be an open set in $\C^{n}$ and $(\mu_\lambda)_{\lambda \in \Lambda}$ be a family of Beltrami forms on $U$. Suppose $\lambda \mapsto \mu_\lambda(z)$ is holomorphic for almost every $z \in U$ , and that there exists a constant $m<1$ such that $||\mu_\lambda||_\infty<m$ for each $\lambda$. For each $\lambda$, extend $\mu_\lambda$ to $ \C$ by $\mu_\lambda=0$ on $\C\setminus U$, and let $f_\lambda: \C\ \longrightarrow \C$ be the unique quasi-conformal homeomorphism such that $f_\lambda^*\mu_0 = \mu_\lambda$, and $\frac{f_\lambda(z)}{z}\longrightarrow 1$ when $|z| \longrightarrow \infty$. Then $(\lambda,z) \mapsto (\lambda,f_\lambda(z))$ is a homeomorphism of $\Lambda  \times \C$ onto itself, and for each $z \in \C$ the map $\lambda \mapsto f_\lambda(z)$ is holomorphic.
\end{enumerate}
\end{thm}
\begin{defn}[Polynomial-like maps]
Given Jordan domains $U,V \subset \C$ with $\overline{U} \subset V$, a polynomial-like map $f: U\longrightarrow V$ is an analytic proper map of finite degree $d$.

The filled Julia set of $f$ is the set 
\begin{align*}
    K_f & = \bigcap _{n\geq 0}f^{\circ n}(U)
\end{align*}

\end{defn}
Given a polynomial $p:\hat{\C} \longrightarrow \hat{\C}$ of degree $d$, we can always find suitable domains $U,V$ such that $\overline{U} \subset V$ and $p\big |_U: U\longrightarrow V$ is polynomial-like of degree $d$.
\begin{defn}[Hybrid Equivalence]
Given two polynomial like maps $f: U \longrightarrow V$ and $g: U' \longrightarrow V'$, we say that $f,g$ are hybrid equivalent if there exists a quasiconformal homeomorphism $\psi: (V,U) \longrightarrow (V',U')$ satisfying $g\circ \psi = \psi \circ f $, with zero dilitation on $K_f$.
\end{defn}
The following theorem is due to Douady and Hubbard, and we shall be using it several times.
\begin{thm}[The Straightening Theorem for polynomial-like maps]
Every polynomial-like map of degree $d$ is hybrid equivalent to a polynomial of degree $d$.
\end{thm}

Our strategy for constructing $\Phi_d$  follows the general layout in \cite{10.1007/BFb0081395}. Given $c \in \mathcal{M}_d$, we will perform a topological surgery in the dynamical plane using the dynamics around one of the $\beta-$ fixed points. At the end of this surgery, we will construct a quasiregular map $g_c$ from a simply-connected Riemann surface $X_1$  to a simply connected Riemann surface $X$ with $\overline{X}_1 \subset X$.

Next, we will show that $g_c$ has an invariant complex structure, and is therefore equivalent to a polynomial-like map of degree $2d+1$. We will finally show that this map is hybrid equivalent to an odd polynomial $p_a$ with $a \in \mathcal{CBO}_d$. The strategy for constructing an inverse for $\Phi_d$ is similar.

All the figures in this paper are illustrations of the case $d=2$. 
\begin{figure}
    \centering
    \includegraphics[scale=0.4]{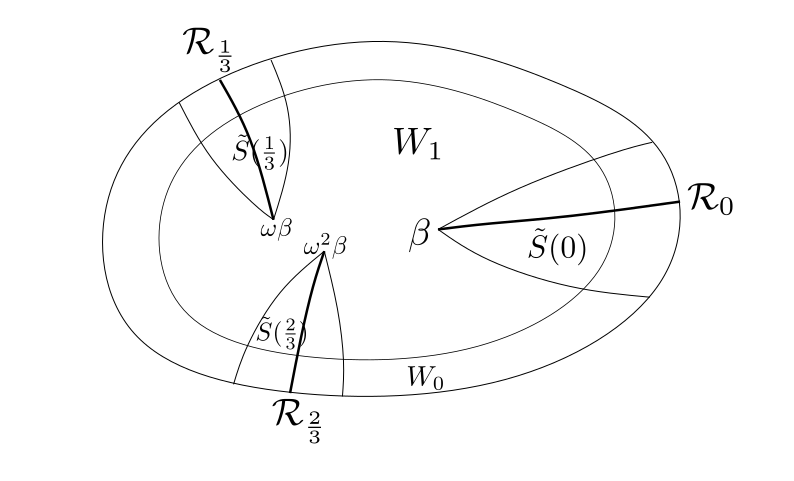}
    \caption{The dynamical plane of $z^3+c$}
    \label{fig:dynplaneinit}
\end{figure}

\begin{figure}
    \centering
    \includegraphics[scale=0.4]{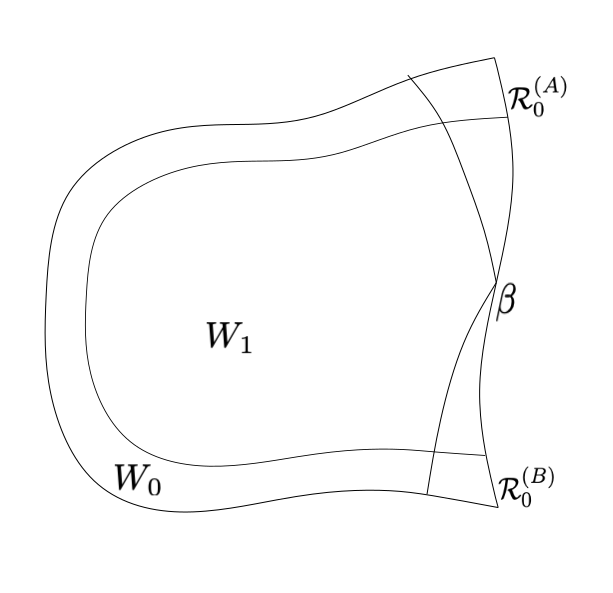}
    \caption{Cutting the dynamical plane of $z^3+c$.}
    \label{fig:my_label}
\end{figure}
\section{Construction of $\Phi_d$}\label{section:defn}
As mentioned before, we proceed along the lines of holomorphic surgery as outlined in \cite{10.1007/BFb0081395}. Let $f_c(z) = z^{d+1}+c$, and $\varphi_c$ be a  B\"{o}ttcher chart at $\infty$ that satisfies $\lim_{z \rightarrow \infty} \frac{\varphi_c(z)}{z} = 1$. In the absence of the latter condition, $\varphi_c$ is unique only  upto multiplication by a $d$th root of unity. By including the condition, we fix a choice of $\varphi_c$ for every $c$ that makes it continuous in the following sense: given $c \in \C$ and $z \in \C$ such that $\varphi_{\tilde{c}}(z)$ is well-defined for $\tilde{c}$ in a neighborhood of $c$, $\tilde{c} \mapsto \varphi_c(z)$ is continuous in $\tilde{c}$.

Now fix $c \in \mathcal{M}_{d+1}$. The dynamical ray $\mathcal{R}_0$ at angle $0$ lands at a fixed point $\beta$ on the dynamical plane of $c$. For a fixed $r>0$, choose  $q,\eta$ such that $q\eta < r$. We will explain how to choose $r$ in the following passages. Let $G_c$ denote the Green's escape rate function.\\
Let \begin{align*}
    W_0 & = \{z: G_c(z) < \eta\}
\end{align*}
and $W_i = f_c^{-1}(W_0)$. We note that 
\begin{align*}
    W_i & = \Biggl\{z: G_c(z) < \frac{\eta}{(d+1)^i}\Biggl\}
\end{align*}
Also define 
\begin{align*}
    \widetilde{S}(0) & = \{\varphi_c^{-1}(e^{s+2\pi i t}): s \in (0,\eta), |t| < qs\}
\end{align*}
We call $\widetilde{S}(0)$ a ``sector'' based at $\beta$. It is invariant under $f_c$, and its inverse image under $f_c$ is a union of similar sectors, each based at a pre-image of $\beta$. More precisely, for $\ell \in \{1,2,...,d\}$, let
\begin{align*}
    \widetilde{S}\Big(\frac{\ell}{d+1}\Big) & = \Big\{\varphi_c^{-1}(e^{s+2\pi i t}): s \in (0,\eta), \Big|t-\frac{\ell}{d+1}\Big| <   qs\Big\} \subset W_0\\
\end{align*}
Then 
\begin{align*}
    f_c^{-1}(\widetilde{S}(0)) & =  \bigcup_{\ell=0}^d \Big(W_1 \cap \widetilde{S}\Big(\frac{\ell}{d+1}\Big) \Big)
\end{align*}
\begin{figure}
    \centering
    \includegraphics[scale=0.4]{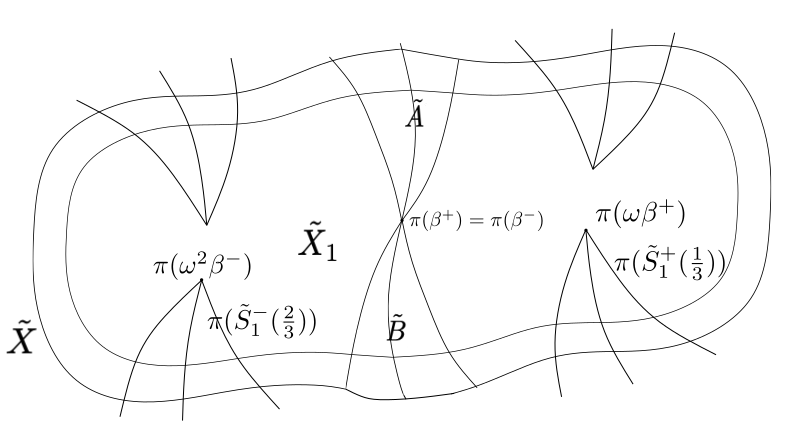}
    \caption{The Riemann Surface $\widetilde{X}$ }
    \label{fig:dynplanefinal} 
\end{figure}
When imposing the condition $q\eta < r$, we choose $r$ small enough so that the sectors  $\widetilde{S}(\frac{\ell}{d+1})$, $\ell=0,1,...,d$, are pairwise disjoint (see Figure~\ref{fig:dynplaneinit}). 

Additonally, form open subsets  $\widetilde{S_i}(0) = W_i \cap \widetilde{S}(0)$ of $\widetilde{S}(0)$. All points in $\widetilde{S}_i(0)$ have escape rates in the interval $\Big(0,\frac{\eta}{(d+1)^{i}}\Big)$, and $f_c$ maps $\widetilde{S}_i(0)$ conformally onto $\widetilde{S}_{i-1}(0)$. By definition, there is a branch of $\log$ that satisfies $$\log \circ \varphi_c(\widetilde{S}(0)) = \{z: Re(z) >0 \text{ and } |Im(z)| < 2\pi q Re(z) \}$$
\subsection{Steps in the definition of $\Phi_d(c)$}\label{sec:defnsteps}
As in \cite[Chapter~II]{10.1007/BFb0081395}, we shall follow this sequence of steps:
\begin{enumerate}
    \item First, we cut along $\mathcal{R}_0$ and glue together two copies of $W_0$, one rotated by $180^\circ$, to get a quotient Riemann surface $\widetilde{X}$
    \item We then construct $f$ on an open subset of $\widetilde{X}$ that is 
    \begin{itemize}
        \item analytic and acts like $f_c$ away from the sectors $\widetilde{S}\big(\frac{\ell}{d+1}\big)$, $\ell \in \{0,1,2,...,d\}$ on both copies of $W_0$
        \item has lines of discontinuities at the two copies of $\mathcal{R}_{\frac{\ell}{d+1}}$ for $\ell \in \{1,2,...,d\}$
    \end{itemize}
    \item We show that by changing $f$ in sectors around these rays, and by modifying the boundary of these sectors, we may construct a quasiregular map $g :X_1 \longrightarrow X$ between simply connected Riemann surfaces with $\overline{X_1} \subset X$, and an almost complex structure $\sigma$ on $X$ that is $g$ invariant. Under the measurable Riemann mapping theorem, there exists a quasi-conformal map $\psi$ such that $\psi \circ g \circ \psi^{-1}$ is analytic. 
    \item Finally, we will apply the straightening theorem to obtain a unique polynomial $p_{a}$ hybrid equivalent to $\psi \circ g \circ \psi^{-1}$.\\
\end{enumerate}

We will now implement these steps one by one.\\
\subsubsection{Cutting along $\mathcal{R}_0$}
Let us cut along $\mathcal{R}_0$. 
In this slit disk, $\widetilde{S}(0)$ is now split into two components $\widetilde{S}_1$ and $\widetilde{S}_2$; we will call the copy of $\mathcal{R}_0$ bounding $\widetilde{S}_1$  $\mathcal{R}_0^{(A)}$ and the one bounding $\widetilde{S}_2$, $\mathcal{R}_0^{(B)}$. Every $x \in \mathcal{R}_0$ now has two copies $x^{(A)}$ and $x^{(B)}$. 

Consider a second copy of this slit $W_0$, and rotate it by $\pi$. We will accent all objects in this (slit) second copy with a $^-$ superscript), and all objects in the original copy with a $^+$ superscript. 
Glue the slit copies  $W_0^+, W_0^-$ together using the following rule:
\begin{align*}
    \forall x \in \mathcal{R}_0, \hspace{10pt}
    x^{(A^+)} &\sim x^{(B^-)}\\
    x^{(B^+)} &\sim x^{(A^-)}
\end{align*}
This gives a quotient map 
\begin{align*}
    \pi : W^+_0 \sqcup W_0^-  \longrightarrow W^+_0 \sqcup W_0^- / \sim
\end{align*}
\noindent This quotient surface can be endowed with a Riemann surface structure that makes $\pi$ analytic away from $\beta^\pm$. We can think of  $\widetilde{X}$ as an open subset of the branched cover over $\C$ corresponding to $w \mapsto w^2+\beta$, and $\pi$ as a branch of $\sqrt{z-\beta}$ on each of the slit copies $W_0^+$ and $W_0^-$.
\begin{figure}
    \centering
    \includegraphics[scale=0.24]{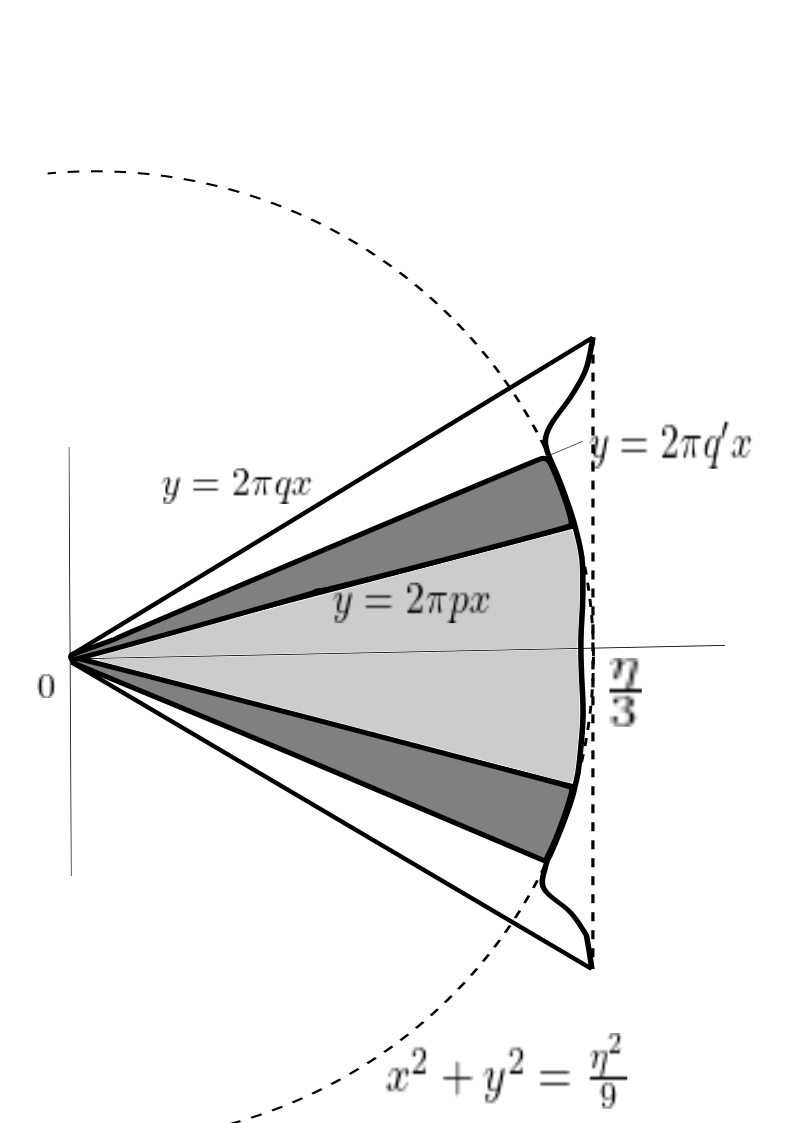}
    \caption{The set $\Delta_q \subset \log \circ \varphi_c(\widetilde{S}_1(0))$. We will eventually define a map that is conformal on the white and lightly shaded regions, and quasiconformal on the darkly shaded region}
    \label{fig:triangle}
\end{figure}
We name this Riemann surface $\widetilde{X}$, and note that $\widetilde{X}$ has smooth boundary.
Let $\widetilde{X_1} = \pi(W^+_1 \sqcup W_1^-)$. Then  $\widetilde{X_1}$ is an open subset of $X$. We also define the ``sectors'' $\widetilde{A}$ and $\widetilde{B}$ as follows: 
\begin{align*}
    \widetilde{A} &= \pi(\widetilde{S}_1^- \cup \widetilde{S}_2^+)\\
    \widetilde{B} &= \pi(\widetilde{S}_2^- \cup \widetilde{S}_1^+)
\end{align*}
See Figure~\ref{fig:dynplanefinal} for an illustration.
\begin{rem}\label{rem:whichbeta}
We could have performed our cut and paste surgery by cutting along $\mathcal{R}_{\frac{j}{d}}$ for any $i \in \{0,1,...,d-1\}$ (the landing points of these rays are precisely the $\beta$ fixed points of $f_c$). To get a continuous embedding $\Phi_d$ of $\mathcal{M}_{d+1}$, however, we will use the same $j$ for all $c \in \mathcal{M}_{d+1}$.
\end{rem}
\subsubsection{Constructing a map $f$ on a subset of $\widetilde{X}_1$}\label{sec:fdefn}
For $z \in \pi(W_1^{\pm})$, define
\begin{align*}
    f(\pi(z)) & = 
    \begin{cases}
    \pi(f_c(z)) & z \not \in \mathcal{R}_0^{A^\pm} \cup \mathcal{R}_0^{B^\pm}\\
    \pi(f_c(z^{(A^+)})) = \pi(f_c(z^{(B^-)}))& z \in (\mathcal{R}_0^{(A^+)} \cap W_1^+) \cup (\mathcal{R}_0^{(B^-)} \cap W_1^-)\\
    \pi(f_c(z^{(B^+)})) = \pi(f_c(z^{(A^-)}))& z \in (\mathcal{R}_0^{(B^+)} \cap W_1) \cup (\mathcal{R}_0^{(A^-)} \cap W_1^-)\\
    \pi(\beta^+) = \pi(\beta^-) & z \in \{\beta,\beta^-\}
    \end{cases}
\end{align*}

For $\ell \in \{1,2,...,d\}$, $f$ is not well defined on $\pi(\mathcal{R}^\pm(\frac{\ell}{d+1}))$ and we cannot extend it over any of these rays continuously since one component of the complement of such a ray in $\pi(\widetilde{S}^\pm_{i-1}(\frac{\ell}{d+1}))$ is mapped to $\widetilde{A}$, and the other is mapped to $\widetilde{B}$. 
\begin{figure}
    \centering
    \includegraphics[scale=0.4]{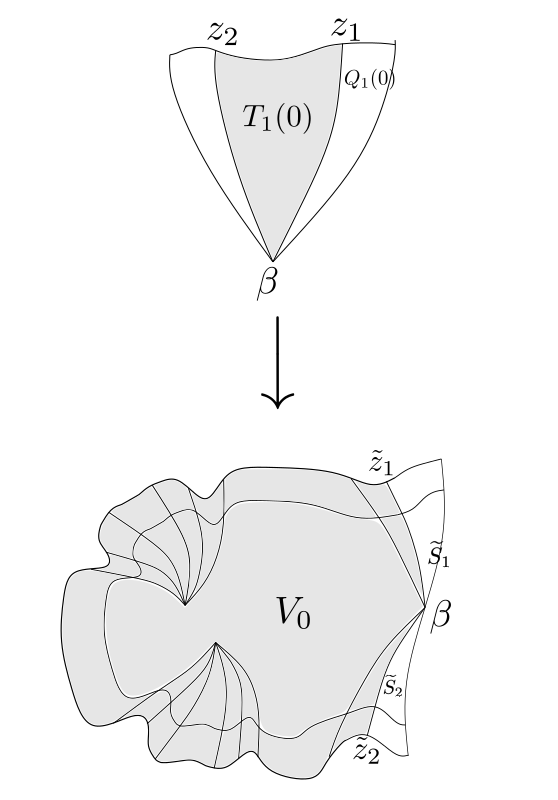}
    \caption{The Riemann map $\widetilde{h}$ maps the lightly shaded region $T_1(0)$ to the region  $V_0$ (the non-white region at the bottom). The darkly shaded region in the bottom figure is $f_c(Y_1(0))$}
    \label{fig:riemannmap}
\end{figure}
However, on the complement in $\widetilde{X}_1$ of the rays above , $f$ is analytic.
\subsubsection{A new map on some sectors} 
By our definition of sectors, note that  $\widetilde{S}\Big(\frac{\ell}{d+1}\Big) = \omega^\ell \tilde{S}(0)$. 

Our strategy will be to produce a quasiregular map $g$ that agrees with $f$ on the complement of the  sets $\pi\Big(\widetilde{S}_1^{\pm}\big(\frac{\ell}{d+1}\big)\Big)$ for $\ell=1,2,...,d$. We let
\begin{align*}
    \widetilde{S}_i\Big(\frac{\ell}{d+1}\Big) & = \widetilde{S}\Big(\frac{\ell}{d+1}\Big) \cap W_i
\end{align*}
$f_c$ maps  $\widetilde{S}_i(\frac{\ell}{d+1})$ conformally to $\widetilde{S}_{i-1}(0)$.

Choose $p,q'$ such that $0<p<q'<q$, and consider the set $\Delta_q$ in $\log$ B\"{o}ttcher coordinates as illustrated in Figure~\ref{fig:triangle}. 
Its boundary is defined so that it is smooth away from the points $0, \log {(\frac{\eta}{d+1})}(1 \pm  2 \pi q i)$, and such that it coincides with an arc of the circle $x^2+y^2=\frac{\eta^2}{(d+1)^2}$ on the two connected regions region bounded by $y=\pm 2 \pi px$ and $y=\pm 2\pi q'x$. 
We will also require the boundary of $\Delta_q$ to be symmetric about the $x$ -  axis in Figure~\ref{fig:triangle}. Additionally, let
\begin{align*}
     \Delta_{q'} & = \Delta_q \cap \{|y| < 2\pi  q'x\}\\
    \Delta_{p} & = \Delta_q \cap \{|y| < 2\pi  px\}
\end{align*}
For $\ell \in \{0,1,...,d\}$, define
\begin{align}
    S_1\Big(\frac{\ell}{d+1}\Big) &= \omega^\ell \varphi_c^{-1} \circ \exp (\Delta_{q}) \label{eqn:s1defn}\\
    Q_1\Big(\frac{\ell}{d+1}\Big)&= \omega^\ell\varphi_c^{-1} \circ \exp (\Delta_{q'}) \label{eqn:q1defn}\\
    T_1\Big(\frac{\ell}{d+1}\Big) &=\omega^\ell \varphi_c^{-1} \circ \exp (\Delta_{p}) \label{eqn:t1defn}\\
    Y_1\Big(\frac{\ell}{d+1}\Big)&=\omega^\ell\varphi_c^{-1} \circ \exp (\Delta_{q}\setminus \overline{\Delta_{q'}}) = S_1\Big(\frac{\ell}{d+1}\Big) \setminus \overline{Q_1\Big(\frac{\ell}{d+1}\Big)}\label{eqn:y1defn}
\end{align}
Clearly,
\begin{align*}
    T_1\big(\frac{\ell}{d+1}\big) \subset Q_1\big(\frac{\ell}{d+1}\big) \subset S_1\big(\frac{\ell}{d+1}\big) \subset \widetilde{S}_1\big(\frac{\ell}{d+1}\big)
\end{align*}
On the slit disk $W_0 \setminus \mathcal{R}_0$, we define $V_0$ as follows:
\begin{align}
    V_0 & = f_c\Big(W_1 \setminus \bigcup_{\ell = 0}^d S_1\Big(\frac{\ell}{d+1}\Big)\Big) \cup f_c(Y_1(0))  \label{eqn:vodefn}
\end{align}
See Figure~\ref{fig:riemannmap} for details.

\begin{figure}
    \centering
    \includegraphics[scale=0.24]{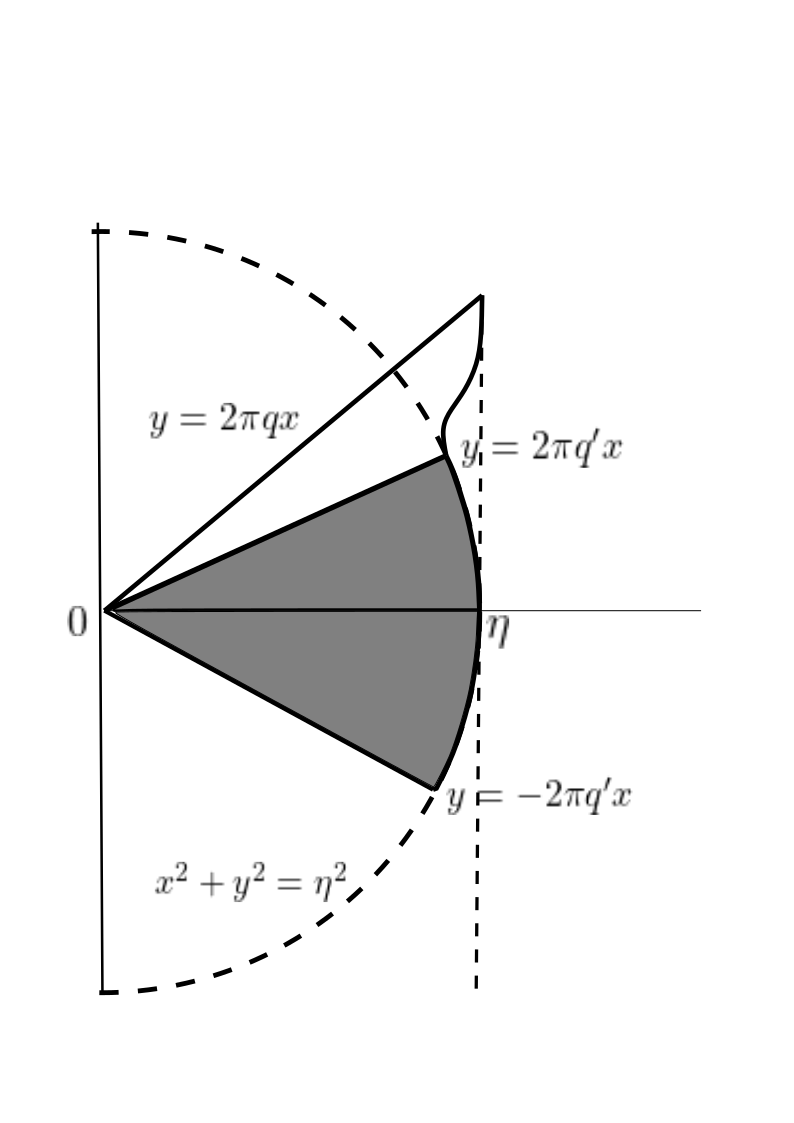}
    \caption{The darkly shaded region above  indicates the set $\Delta$- the image under $\hat{h}$ of a connected component of $\Delta_{q'} \setminus \overline{\Delta_{p}}$}
    \label{fig:Delta}
\end{figure}
By the Riemann mapping theorem, there exist analytic maps $\widetilde{k}: \mathbb{D} \longrightarrow V_0$ and  $\widetilde{m}: T_1(0)  \longrightarrow \mathbb{D}$ such that
\begin{align*}
    \widetilde{k}(\widetilde{m}( \beta))&= \beta\\
    \widetilde{k}(\widetilde{m}(z_1)) &=  \widetilde{z}_1\\
    \widetilde{k}(\widetilde{m}(z_2))&=\widetilde{z}_2
\end{align*}
Let 
$\widetilde{h} = \widetilde{k}\circ \widetilde{m}$. The map $\widetilde{h}:T_1(0) \longrightarrow V_0$ is the unique analytic function that sends the triple $(\beta,z_1,z_2)$ to $(\beta, \widetilde{z}_1,\widetilde{z}_2)$ (see Figure~\ref{fig:riemannmap}). $\partial V_0, \partial T_1(0)$ are quasi-circles. Therefore, $\widetilde{k},\widetilde{m}$ extend to quasisymmetric maps on the boundaries of their respective domains. 
\noindent Furthermore, by \cite[Chapter~3.4, Exercise 1]{pommerenke},
\begin{align*}
    \widetilde{k}(z) &= \beta + a_k(z-1)^{2-4q'} + O(|z-1|^{2-4q' + (2-4q')\gamma_k})\\
        \widetilde{m}^{-1}(z) &= \beta + a_m(z-1)^{4p} + O(|z-1|^{4p + 4p\gamma_k})
\end{align*}
for some $\gamma_k, \gamma_m \in (0,1), a_k, a_m \in \C$.

It then follows that
\begin{align*}
    \widetilde{h}(z) = \beta + a_h\big(z-\beta\big)^{\frac{1-2q'}{2p}} + O\Big(\big|z-\beta\big|^{\frac{1-2q'}{2p} + \big(\frac{1-2q'}{2p}\big) \gamma_h}\Big)
\end{align*}
for some $\gamma_h \in (0,1), a_h \in \C$. Since $\varphi_c$ does not distort angles, conjugating $\tilde{h}$ by $\varphi_c$ should not change this equation, and we have the following:
\begin{prop}\label{boundarybeh}
For all $z \in \exp(\Delta_p) = \varphi_c(T_1(0))$,
\begin{align}
    \varphi_c \circ \widetilde{h} \circ \varphi_c^{-1}(z) = 1 + a'\big(z-1\big)^{\frac{1-2q'}{2p}} + O\Big(\big|z-1\big|^{\frac{1-2q'}{2p}+ (\frac{1-2q'}{2p})\gamma}\Big) \label{eqn1}
\end{align}
for some $\gamma \in (0,1), a' \in \C \setminus \{0\}$. 
\end{prop}
Let $G$ be a connected component of $\Delta_{q'} \setminus \overline{\Delta_p}$, say the one bounded by $y=2\pi q'x$ and $y=2\pi  p x$. The polynomial $f_c$ induces the map $\mu_{d+1}(z)= (d+1)z$ on the part of its boundary where $y = 2\pi q'x$. The map $\hat{h} = (\log \circ  \varphi_c) \circ \widetilde{h} \circ ( \varphi_c^{-1}  \circ \exp)$ extends to a continuous map on the part of the boundary where $y= 2\pi px$.

Let $\Delta$ be the set $\Delta_{q'} \cap\{ x^2+y^2< \eta^2\}$(see Figure~\ref{fig:Delta} for an illustration of $\Delta$). $S(0):=\varphi_c^{-1} \circ \exp(\Delta)$ is an open subset of $\widetilde{S}(0)$. For $\ell \in \{1,2,...,d\}$, let $S(\frac{\ell}{d+1}) = \omega^\ell S(0)$. The sector $S\big(\frac{\ell}{d+1}\big)$ contains $S_1\big(\frac{\ell}{d+1}\big)$.\\

The following is a crucial lemma.
\begin{figure}
    \centering
    \includegraphics[scale=0.5]{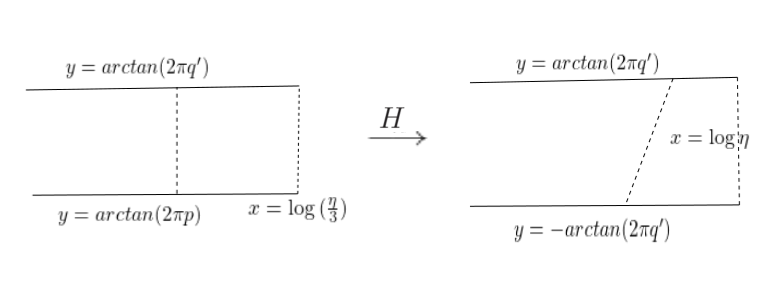}
    \caption{The map $H$ is defined by mapping vertical lines to lines joining the images of their endpoints}
    \label{fig:maponstrips}
\end{figure}
\begin{lemma}\label{lemma:quasiexists}
There exists a quasiconformal map from $G$ to $\Delta$ that restricts to $\mu_{d+1}$ on one boundary, and to $\hat{h}$ on the other boundary.
\end{lemma}
\begin{proof}
$G$ has a positive angle at the vertex $0$.
In $\log$ coordinates, $G' = \log G$ is a half-infinite horizontal strip with $\hat{\mu}_{d+1}(z) = z+\log (d+1)$ induced by $\mu_{d+1}$ on the part of the boundary where $y = \arctan(2 \pi q')$, and $H(z) = \log \hat{h}(e^z)$ on the part of the boundary where $y = \arctan(2 \pi p)$.

We will interpolate between $\hat{\mu}_{d+1}$ and $H$ by mapping vertical lines in $G'$ to lines joined by the images of the endpoints. If we can show that these image lines have uniformly bounded slope, the resulting map will be quasiconformal. We explain this in detail below.

Set $\theta_p = \arctan(2 \pi p), \theta_{q'} = \arctan(2\pi q').$ We will define $H$ on $\overline{G'}$ by extending along vertical lines:
\begin{align*}
    H\Big(x+i\big((1-t)\theta_p +t \theta_{q'}\big)\Big) & = (1-t)H\big(x + i\theta_p\big) + t\hat{\mu}_{d+1}\big(x+i \theta_{q'}\big)  
\end{align*}

\begin{prop}\label{boundedlines}
There exists $R > 0$ such that for all $z \in \{Im(z) = \arctan(2\pi p)\} \cap \overline{G'}$, 
\begin{align}
    |H(z) - z| \leq R \label{eqn2}
\end{align}
\end{prop}
\begin{proof}
We will prove this by showing that both $z \mapsto z - H(z)$ and $z \mapsto H(z)-z$ are bounded above.

Suppose $z-H(z)$ is not bounded above, then for each natural number $n$, there exists $z_n$ such that 
\begin{align*}
    z_n - H(z_n) &> n
\end{align*}
and upto a subsequence, the $z_n$  tend to $ -\infty$.
But this implies that 
\begin{align*}
    |\hat{h}(u_n)|&< \frac{|u_n|}{e^n}
\end{align*}
where $u_n = \exp z_n$. Furthermore,
\begin{align*}
    Re(\hat{h}(u_n)) &\leq |\hat{h}(u_n)| \\& \leq  \frac{|u_n|}{e^n} = \frac{1}{e^n}\sqrt{Re(u_n)^2(1+4\pi^2 p^2)} \\& \leq C\frac{|Re(u_n)|}{e^n} =C\frac{Re(u_n)}{e^n}
\end{align*}
for some constant $C>0$. \\

Set $w_n = \exp {u_n}$, and note that $\exp \hat{h}(u_n) = \varphi_c \circ \widetilde{h}\circ \varphi_c^{-1}(w_n)$. Thus
\begin{align}
    |\varphi_c \circ \widetilde{h}\circ \varphi_c^{-1}(w_n)| < b|w_n|^{\frac{1}{e^n}} \label{eqn3}
\end{align}
for some $b>0$. 

But Equation~\ref{eqn3} implies that $|\varphi \circ \widetilde{h}\circ \varphi_c^{-1}(w_n)|$ converges much faster to $1$ than allowed by Equation~\ref{eqn1}, and forms a contradiction. This proves that $z - H(z)$ is bounded above.

Similarly, suppose $H(z)-z$ is not bounded above as $z \rightarrow -\infty$,  there exists a sequence $z_n \rightarrow -\infty$ such that 
\begin{align*}
    H(z_n) - z_n &> n
\end{align*}
Thus
\begin{align*}
    |\hat{h}(u_n)| &> |u_n|e^n
\end{align*}
Consequently, we have 
\begin{align*}
    Re\hat{h}(u_n) &= \frac{|\hat{h}(u_n)| }{\sqrt{1+4\pi^2q'^2}} \\
    &> \frac{|u_n|e^n}{\sqrt{1+4\pi^2q'^2}} \\
    &= \frac{\sqrt{1+4\pi^2p^2}}{\sqrt{1+4\pi^2q'^2}}Re(u_n)e^n
\end{align*}
But this gives us 
\begin{align*}
    |\varphi_c \circ \widetilde{h}\circ \varphi_c^{-1}(w_n)| > \iota |w_n|^{e^n}
\end{align*}
for some constant $\iota >0$, which also contradicts Equation~\ref{eqn1}.

This proves Proposition~\ref{boundedlines}.
\end{proof}
\begin{figure}
    \centering
    \includegraphics[scale=0.4]{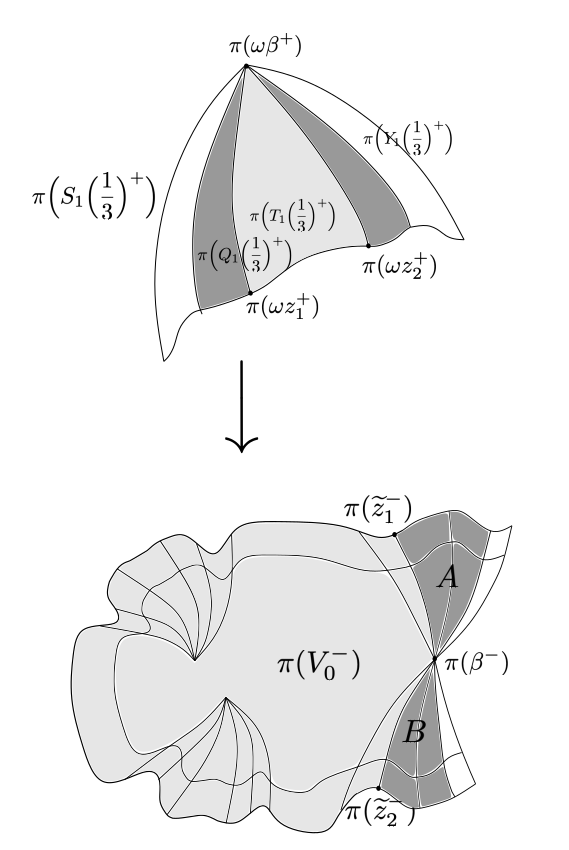}
    \caption{An illustration of $h$ on $ \pi\big(S_1\big(\frac{1}{3}\big)^+\big)$; the dark components  above are quasiconformally mapped to the dark components below, the lightly shaded region maps by the Riemann map $\omega^{-1}\widetilde{h}$, and the white region maps by $f$ }
    \label{fig:quasimap}
\end{figure}
It is clear that $H$ interpolates between the maps on the two horizontal boundaries, and that $H(G' \cap \{Re(z) = \log(\frac{\eta}{3})\}) = \{\log \eta + iy : |y| < \arctan(2 \pi q')\}$. Furthermore, $H$ is a quasiconformal map whose dilitation is bounded above by some $M\geq 1$ (see Figure~\ref{fig:maponstrips}): this is because vertical lines in the domain are mapped by $H$ to lines whose slopes are bounded below by some uniform constant, by Equation~\ref{eqn2}. 
We conjugate $H$ by the exponential map to obtain a quasiconformal map $\hat{h}$ from $G$ that satisfies the properties in the statement of Lemma~\ref{lemma:quasiexists}.
\end{proof}

$G$ could be taken to be either connected component of $\Delta_{q'} \setminus \overline{\Delta_{p}}$. On the dynamical plane, it corresponds to a component $\mathcal{G}$ of $Q_1(0) \setminus \overline{T_1(0)}$. We shall henceforth denote the copy of $S(0)$ in $\widetilde{A}$ as $A$, and the copy in $\widetilde{B}$ as $B$. With this in mind, we will take $h_{\mathcal{G},A}$ to mean the map $ (\omega^\ell \circ \varphi_c^{-1} \circ \exp) \circ \hat{h} \circ (\log \circ  \varphi_c \circ \omega^{-\ell})$ from the component $\mathcal{G}$ of $Q_1\big(\frac{\ell}{d+1}\big) \setminus T_1\big(\frac{\ell}{d+1}\big)$ to $A$, and $h_{\mathcal{G},B}$ to mean the same map, but from $\mathcal{G}$ to $B$. We will use the same names for the extended maps from $\overline{\mathcal{G}}$.
\subsubsection{Constructing a quasiregular map $g$}
Let $S$ be a sector of the form $\pi\Big(S^+_1\big(\frac{\ell}{d+1}\big)\Big)$, where $\ell \in \{1,2,...,d+1\}$. The map $f$ has a line of discontinuities in $S$ along the ray $\pi\Big(\mathcal{R}^+_{\frac{\ell}{d+1}}\Big)$- on one side of this ray, $f$ maps into $\widetilde{A}$ and approaches $\pi(\mathcal{R}_0^{\widetilde{A}})$, whereas on the other side, $f$ maps into $\widetilde{B}$ and  approaches $\pi(\mathcal{R}_0^{\widetilde{B}})$. Define a map $h$ on $\overline{S}$ as follows:
\begin{itemize}
    \item On $\pi\Big(\overline{T^+_1\big(\frac{\ell}{d+1}\big)}\Big)$, let $h(\pi(z^+)) = \pi(\widetilde{h}(\omega^{-\ell}z)^-) \in \pi(V_0^-)$.
    \item On $\pi\Big(\overline{Y_1^+\big(\frac{\ell}{d+1}\big)}\Big)$, let $h(\pi(z^+)) = f(\pi(z^+))$, where $f$ is defined as in Section~\ref{sec:fdefn}.
    \item On the connected component $\mathcal{G}$ of $\pi\Big(Q^+_1\big(\frac{\ell}{d+1}\big) \setminus T^+_1\big(\frac{\ell}{d+1}\big)\Big)$ with $\ell=1,2,...,d$ part of whose boundary $f$ maps into $\partial \widetilde{A}$, let $h(\pi(z^+)) = h_{\mathcal{G},A}(z)$.
    \item On the connected component $\mathcal{G}$ of $\pi\Big(Q^+_1\big(\frac{\ell}{d+1}\big) \setminus T^+_1\big(\frac{\ell}{d+1}\big)\Big)$ with $\ell=1,2,...,d$ part of whose boundary  $f$ maps into $\partial \widetilde{B}$, let $h(\pi(z^+)) = h_{\mathcal{G},B}(z)$.
\end{itemize}
\begin{figure}
\centering
\begin{subfigure}[b]{0.65\textwidth}
    \includegraphics[width=\textwidth]{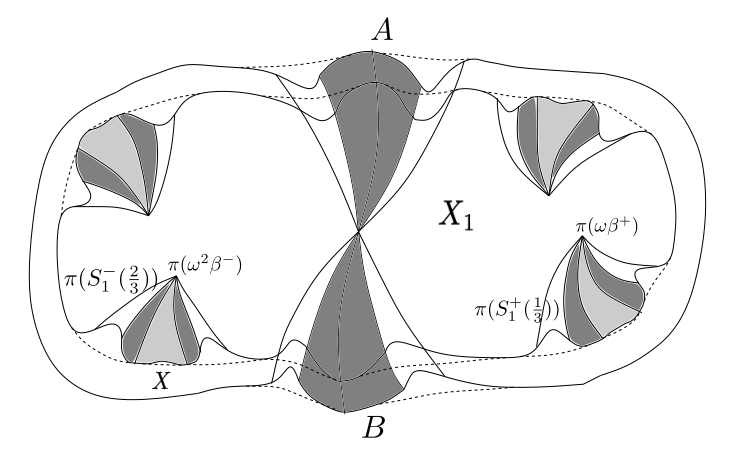}
    \caption{The Riemann surfaces $X_1$, $X$}
    \label{fig:dynplanemodified1}
\end{subfigure}\\
\begin{subfigure}[b]{0.65\textwidth}
    \includegraphics[width=\textwidth]{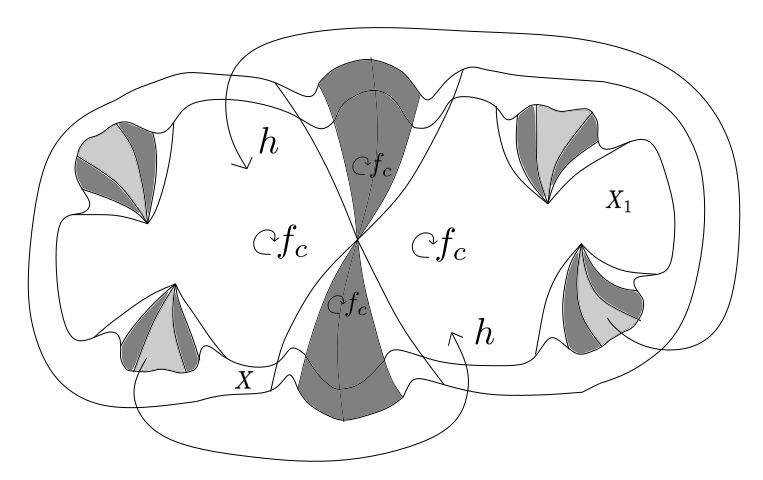}
    \caption{The quasiregular map $g$}
    \label{fig:dynplanemodified2}
\end{subfigure}
\label{fig:dynplanemodified}
\caption{The dynamics of a quasiregular model for $\Phi_d(c)$}
\end{figure} 

The map $h$ so defined is a quasiconformal homeomorphism from $S$ to $\pi(V_0^-) \cup A \cup B$ (see Figure~\ref{fig:quasimap} for an illustration of $h$ on $\pi\Big(S\big(\frac{1}{3}\big)^+\Big)$ when $d=2$), and  restricts to an analytic map on $\pi\Big(T_1\Big(\frac{\ell}{d+1}\Big)\Big)$.

Furthermore, the latter set has smooth boundary at the points $\pi(\widetilde{z}_1^-)$ and $\pi(\widetilde{z_2}^-)$ in Figure~\ref{fig:quasimap}: consider $\pi(\widetilde{z}_1^-)$ for instance. In B\"{o}ttcher coordinates, the boundary in a neighborhood of  $\widetilde{z}_1$ looks like $f_c(\varphi_c^{-1} \circ \exp(\gamma))$, where $\gamma$ is a neighborhood of the boundary of $\Delta_{q}$ at the point $\log \circ \varphi_c(z_1)$; $\gamma$ is clearly smooth. 

On $\pi\Big(S_1^{-}\big(\frac{\ell}{d+1}\big)\Big)$, we define $h$ the same way, except with the following change: on $\pi\Big(\overline{T_1^-\big(\frac{\ell}{d+1}\big)}\Big)$, let $h(\pi(z^-)) = \pi(\widetilde{h}(\omega^{-\ell}z)^+) \in \pi(V_0^+)$. This $h$ is a quasiconformal hoemoemorphism from $S$ to $\pi(V_0^+) \cup A \cup B$.
Finally, we construct a quasiregular map on newly defined subsets of $\widetilde{X}_1, \widetilde{X}$. \\
Let 
\begin{align*}
    X & = \pi(V_0^+) \cup A \cup B \cup \pi(V_0^-)
\end{align*}
Also let $X_1$ be the subset of $\widetilde{X_1}$ where all sectors of the form $\pi\Big(\widetilde{S}^\pm_1\big(\frac{\ell}{d+1}\big)\Big)$, $\pi\Big(\widetilde{S}^-_1\big(\frac{\ell}{d+1}\big)\Big)$ are replaced by $\pi\Big(S_1^\pm\big(\frac{\ell}{d+1}\big)\Big)$ for  $\ell=1,2,..,d$, and let $X$ be the open subset of $\widetilde{X}$ where $\widetilde{A}$ and $\widetilde{B}$ are replaced by $A$, $B$ respectively. See Figure~\ref{fig:dynplanemodified1} for details.
\noindent Clearly, $X_1$ is an open subset of the Riemann surface $X$ compactly contained in $X$. We define
\begin{align*}
    g&: X_1 \longrightarrow X\\
    g(\pi(z)) & = \begin{cases}
    f(\pi(z)) & z \in \Big(W^+_1 \setminus \bigcup _\ell S^+_1\big(\frac{\ell}{d+1}\big)\Big) \cup \Big(W^-_1 \setminus \bigcup_\ell S^-_1\big(\frac{\ell}{d+1}\big)\Big),\ell=1,2,...,d \\
    h(\pi(z)) & z \in S^\pm_1\big(\frac{\ell}{d+1}\big), \ell=1,2,...,d\\
    \end{cases}
\end{align*}
See Figure~\ref{fig:dynplanemodified2} for an illustration of $g$.\\ 

$g$ is quasiregular. Furthermore, any $g$ orbit visits $\pi\Big(Q^+_1\big(\frac{\ell}{d+1}\big) \setminus T^+_1\big(\frac{\ell}{d+1}\big)\Big)$ or $\pi\Big(Q^-_1\big(\frac{\ell}{d+1}\big) \setminus T^-_1\big(\frac{\ell}{d+1}\big)\Big)$ at most once, and  these are the only regions where $g$ is not analytic. 
We will use this fact to define a new complex structure $\sigma$ (given by an ellipse field $E_x$ for $x \in X$) by setting
\begin{itemize}
    \item $E_x = \mathbb{S}^1$ if $x \in X \setminus X_1$ or if the orbit of $x$ never visits  $\pi\Big(Q^+_1\big(\frac{\ell}{d+1}\big)\setminus T^+_1\big(\frac{\ell}{d+1}\big)\Big)$ or $\pi\Big(Q^-_1\big(\frac{\ell}{d+1}\big) \setminus T^-_1(\frac{\ell}{d+1}\big)\Big)$
\item $E_x = (T_xg)^{-1}(\mathbb{S}^1)$ for $x \in \pi\Big(Q^+_1\big(\frac{\ell}{d+1}\big)\setminus T^+_1\big(\frac{\ell}{d+1}\big)\Big)$ or $x \in \pi\Big(Q^-_1\big(\frac{\ell}{d+1}\big) \setminus T^-_1\big(\frac{\ell}{d+1}\big)\Big)$ for some $\ell \in \{1,2,...,d\}$
    \item $E_x = (T_x g^{n})^{-1}(E_{g^n(x)})$ if $g^n(x)$ is the first point in the $g-$orbit of $x$ that is in one of the regions above.
\end{itemize}
The complex structure $\sigma$ thus defined has bounded dilitation, and $g^*\sigma = \sigma$.
\subsubsection{Obtaining a polynomial}
Define the map $\tau: X \longrightarrow X$ by sending $\pi(z^+)$ to $\pi(z^-)$. $\tau$ satisfies $\tau^*\sigma = \sigma$, $\tau(X_1) = X_1$, and $\tau^{\circ 2} = id$. We note that
\begin{align*}
    g \circ \tau = \tau \circ g 
\end{align*}

We find an integrating map $\psi$ for $\sigma$ sending $\pi(\beta^\pm)$ to $0$, and satisfying $\frac{\psi(z)}{z} \longrightarrow 1$ as $z \longrightarrow \infty$. 
The map $G= \psi \circ g\circ \psi^{-1} : U' \longrightarrow U$ is polynomial-like, and has two critical points with local degree $d+1$. $\kappa = \psi \circ \tau \circ \psi^{-1}$ an analytic involution of $U$, and commutes with $G$ on $U$. We can further conjugate by a Riemann map taking the pair $(U,0)$ to $(\D,0)$. By the Schwarz lemma, we can assume without loss of generality that $\kappa\big |_U(z) = -z$; in particular, $\kappa$ has a global extension.

$G$ is hybrid equivalent to a degree $2d+1$ polynomial $p$ with two critical points by the straightening theorem. We may choose this hybrid equivalence $h:\C \longrightarrow \C$ such that $h(0) = 0$. Then
$\delta = h \circ \kappa\circ  h^{-1}$ is an affine map of $\C$ with $\delta(0) = 0, \delta^{\circ 2}  = id, \delta \neq id$. Therefore, $\delta(z) = -z$. $p$ commutes with $\delta$, and can now be normalized to the form $a{\displaystyle \int_{0}^{z} } \Big(1-\frac{w^2}{d}\Big)^ddw$ for a unique $a \in \C^*$. The choice of $a$ does not depend on our initial choice of $q,\eta,\psi$ ($h$ is determined by  $c,p,q,q',\eta$) - we can show that different choices give rise to hybrid equivalent polynomials.
\subsection{The image of $\Phi_d$}
Clearly, $\Phi_d(\mathcal{M}_{d+1}) \subset \mathcal{CBO}_d$. Let $a=\Phi_d(c)$. By our construction, $0$ is a fixed point of $p_{a}$ belonging to the Julia set, and it disconnects the Julia set into two components. 

Under our surgery, the original dynamical ray $\mathcal{R}_0$ landing at $\beta$ gets transformed into an arc $\Gamma$ from $0$ to $\infty$ in the dynamical plane of $p_{\Phi_d(c)}$ whose interior is contained in the escaping set. In the monic representation $P_{s(a)}$ of $p_a$, $\Gamma$ has the same access as $\mathcal{R}_{s(a)}(0)$. The union  $\Gamma \cup -\Gamma$, and indeed $\mathcal{R}_0(s(a)) \cup \mathcal{R}_\frac{1}{2}(s(a))$, separates the orbits of the two critical points of $P_{s(a)}$. That is, $\Phi_d(\mathcal{M}_{d+1}) \subset \mathcal{CBO}^{(+,-)}_d$. We will show in Section~\ref{section:injectivity} that the image under $\Phi_d$ is equal to this set.
\section{Continuity of $\Phi_d$}\label{section:continuity}
To show continuity of $\Phi_d$, we will follow the strategy laid out in \cite[Chapter~II.8]{10.1007/BFb0081395}, and show it separately when $c$ is on the boundary, or in the interior of $\mathcal{M}_{d+1}$.  Throughout this section, we shall index all sets and functions in Section~\ref{section:defn} in constructing $\Phi_d(c)$ by the subscript $c$. For example, the projecton $\pi$ is  referred to as $\pi_c$, the quasiregular map $g$ as $g_c$, the domain of $g_c$ as $(X_1)_c$ and so on.
\begin{lemma}\label{lemm:hybridimpliesaffine}
If $p_a,p_{a'}$ with $a,a' \in \mathcal{CBO}_d$ are hybrid equivalent, then they are affine conjugate.
\end{lemma}
\begin{proof}
This follows from \cite[Chapter~I.6, Corollary~2]{Douady1985}.
\end{proof}
\subsection{The interior case}

If $c \in \mathcal{M}_{d+1}^\circ$, then proof is based on the proof of \cite[Chapter~II.5, Proposition~12]{Douady1985}. 
\begin{defn}\label{defn:analyticfamily}
    Given 
    $f_\lambda : U_\lambda' \longrightarrow U_\lambda$, for $\lambda \in \Lambda$, let
    \begin{align*}
        \mathcal{U}' & = \{(\lambda, z)| z \in U_\lambda'\}\\
        \mathcal{U} & = \{(\lambda, z)| z \in U_\lambda\}
    \end{align*}
    and define $f: \mathcal{U}' \longrightarrow \mathcal{U}$ as $f(\lambda, z) = f_\lambda(z)$. If 
    \begin{enumerate}
    
        \item $\mathcal{U}',\mathcal{U}$ are homeomorphic over $\Lambda$ to $\Lambda \times \D$.
        \item projection of $\overline{\mathcal{U}'}$ in $\mathcal{U}$ to $\Lambda$ is proper
         \item $f$ is holomorphic and proper
       
    \end{enumerate}
    then $f_\lambda$ is called an analytic family.
    \end{defn}

Let us go back to the construction of $\Phi_d(c)$ from $f_c$. We first construct a quasiregular map $g_{c}: (X_1)_{c} \longrightarrow X_{c}$. This map is built from $f_c$ away from certain escaping sectors, and from the Riemann map $h_{c}$ on other sectors. Then we find an invariant complex structure $\sigma_c$ for $g_{c}$ and find integrating maps $\psi_c$. This gives us the polynomial-like family $G_{c}: U'_{c} \longrightarrow U_{c}$.
        \begin{prop}\label{prop:analyticfam}
     On a connected component $\Lambda$ of $\mathcal{M}^\circ_{d+1}$, $(c,z) \mapsto (c,G_{c}(z))$ is an analytic family of structurally stable polynomials.
     \end{prop}
     \begin{proof}
We show that $G_{c}$ satisfies the three properties of Definition~\ref{defn:analyticfamily}. 
   \begin{enumerate}
       \item $U'_{c},  U_{c}$ are homeomorphic to $\D$ and $c' \mapsto {U}'_{c}, c \mapsto U_{c}$ are both continuous maps in the Hausdorff topology 
       \item Let $\Pi$ be this projection. Given any compact set $K$ in $\Lambda$, and a sequence $(c_n, z_n) \in  \Pi^{-1}(K)$, upto a subsequence, $c_n \longrightarrow c \in K$. We note that $z_n \in \overline{U'_{c_n}}$. $\overline{U'}_{c_n} \longrightarrow \overline{U'}_c$, and hence, there exists a sequence $\widetilde{z}_n \in \overline{U'}_c$ such that $|z_n - \widetilde{z}_n| \longrightarrow 0$. Since $\overline{U'}_c$ is compact, upto a subsequence, $\widetilde{z}_n \longrightarrow z \in \overline{U'}_c$. So $z_n \longrightarrow z$ upto the same subsequence. This shows that $\Pi^{-1}(K)$ is compact.
       \item By \cite{mane1982dynamics}, every parameter $c \in \Lambda$ is structurally stable.   More particularly, given $c \in \Lambda$, there exists a holomorphic motion $L: \Lambda \times \hat{\C} \longrightarrow \hat{\C}$ such that $L_c = id$, and for all $\widetilde{c} \in \Lambda$,  $L_{\widetilde{c}}$ is quasiconformal and satisfies  $L_{\widetilde{c}}\circ f_c\circ L_{\widetilde{c}}^{-1} = f_{\widetilde{c}}$. 
       \begin{figure}
\centering
    \includegraphics[scale=0.4]{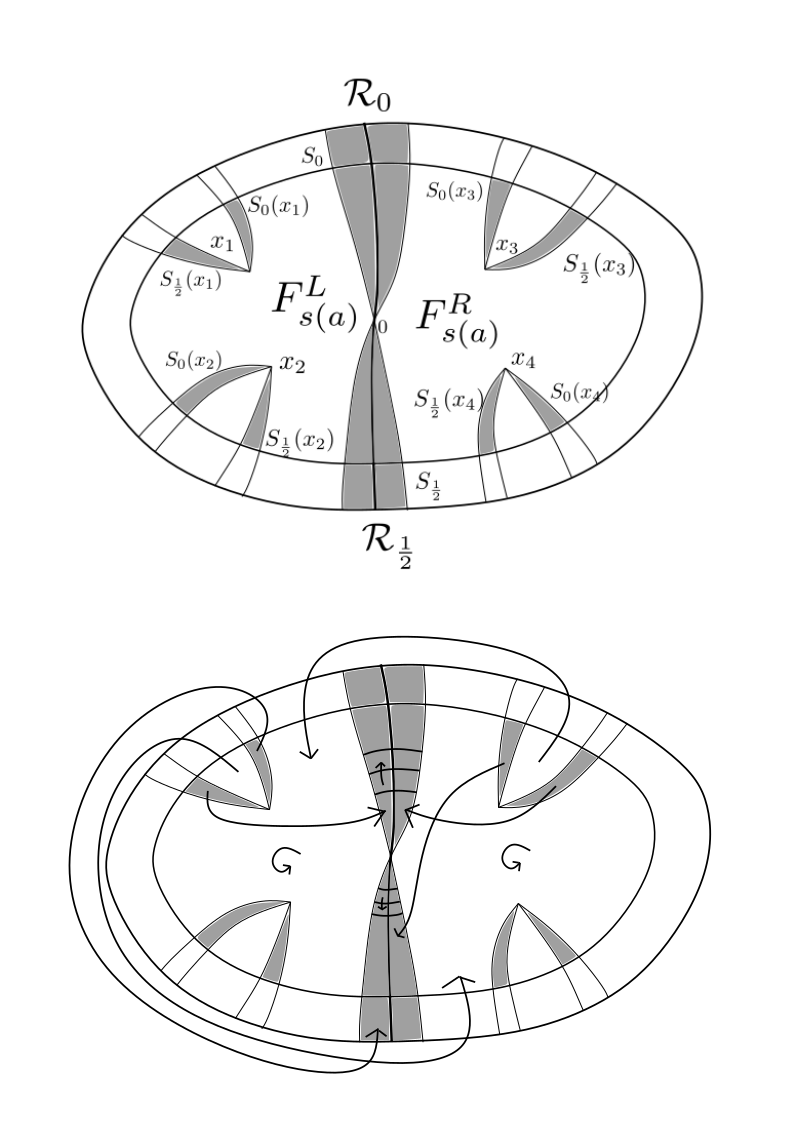}

\caption{Dynamics of  $P_{s(a)}$ for $a \in \mathcal{CBO}_d^{(+,-)}$}
    \label{fig:inversedefn}
\end{figure}
       But this also means that $h_{\widetilde{c}} =L_{\widetilde{c}}\circ h_c\circ L_{\widetilde{c}}^{-1}$ on $\varphi^{-1}_{\widetilde{c}} \circ  \exp (\Delta_p)$, and by definition of the quasiconformal extension,  $h_{\widetilde{c}} =L_{\widetilde{c}}\circ h_c\circ L_{\widetilde{c}}^{-1}$ on the sector $S_1\big(\frac{\ell}{d+1}\big)$ in the dynamical plane of $f_{\widetilde{c}}$.
     Therefore, $g_{c}$, and consequently, $\sigma_c$ depend analytically on $c$. By the measurable Riemann mapping theorem, the integrating maps $\psi_c$ depend holomorphically on $c$. 
     
     For a fixed $z \in U'_c$, when $\widetilde{c}$ is close to $c$, $G_{\widetilde{c}}(z)$ is well-defined, and $\widetilde{c} \mapsto G_{\widetilde{c}}(z) = \psi_{\widetilde{c}} \circ g_{\widetilde{c}}\circ \psi_{\widetilde{c}}^{-1}(z)$ is holomorphic in $\widetilde{c}$. 
     For a fixed $c$, $z \mapsto G_c(z)$ is holomorphic by definition. Thus  $G_c(z)$ is holomorphic in both $c$ and $z$; by Hartog's theorem, it is holomorphic as a function of $(c,z)$. Proof that $G_c(z)$ is proper is similar to Point~2 above.
   \end{enumerate}
     \end{proof}
In Proposition~\ref{prop:analyticfam}, we showed that $(c,z) \mapsto G_c(z)$ is an analytic family over every connected component $\Lambda$ of $\mathcal{M}^{\circ}_{d+1}$. Given $c \in \Lambda$, let us pick the hybrid equivalence $k_{{c}}$ conjugating $G_c(z)$ to a polynomial in such a way that it fixes $0$ and satisfies $k_{{c}}(z)/z \longrightarrow 1$. Then, by \cite[Chapter~II.5, Proposition~12]{Douady1985}, the polynomials $k_{{c}}\circ G_{{c}}\circ k_{{c}}^{-1}$ form a continuous family over $\Lambda$. As proved in Section~\ref{section:defn}, these are affine conjugate to  bicritical odd polynomials, and their critical points vary continuously with respect to  ${c}$. Hence there exists a continuous family of scaling maps $M_{{c}}$ that map these critical points to $\pm \sqrt{d}$.  But then
 \begin{align*}
     p_{\Phi_d(\widetilde{c})} = M_{\widetilde{c}}\circ k_{\widetilde{c}}\circ G_{\widetilde{c}}\circ k_{\widetilde{c}}^{-1}\circ M_{\widetilde{c}}^{-1},
 \end{align*} is clearly continuous in $\widetilde{c}$.
\subsection{The Boundary case}
The following lemma and its proof are similar to \cite[Chapter~II.8, Lemma~3]{10.1007/BFb0081395}.
\begin{lemma}\label{lemm:quasi}
If $p_a$ and $p_{a'}$ are quasiconformally equivalent via $\psi$, with $a \in \partial \mathcal{CBO}_d$, such that $\psi$ satisfies the conditions below:
\begin{align*}
    \psi(0) &= 0\\
    \psi(\sqrt{d}) & =\sqrt{d}\\
\lim_{z \longrightarrow \infty}\frac{\psi(z)}{z} &=1
\end{align*}
then $a=a'$.
\end{lemma}
\begin{proof}
We first note that any $\psi$ as above also satisfies $\psi(-\sqrt{d}) = -\sqrt{d}$. If $K_{p_a}$ has measure $0$, then $\psi$ is a hybrid equivalence and the result follows. 

Otherwise, our strategy is to build a hybrid equivalence between the two polynomials, similar to \cite[Chapter~I.6, Corollary~2]{Douady1985} and use Lemma~\ref{lemm:hybridimpliesaffine}. Consider the Beltrami form $\mu = \frac{ \overline{\partial} \psi}{\partial \psi}$ and let $\mu_0$ be the form that agrees with $\mu$ on $K_{p_a}$ and equals $0$ on $\C \setminus K_{p_a}$. Set $k = ||\mu_0||_\infty$. 

We note that $k<1$. By the measurable Riemann mapping theorem, for every $t \in \D_{\frac{1}{k}}$, there exists a unique quasiconformal homeomorphism $\psi_t : \C \longrightarrow \C$ such that 
\begin{align*}
    \frac{\overline{\partial }\psi_t}{\partial \psi_t} & = t\mu_0\\
    \psi_t(0) &= 0\\
    \lim_{z \rightarrow \infty}\frac{\psi_t(z)}{z}& = 1
\end{align*}
 We note that $\kappa_t (z)= \psi_t(-\psi_t^{-1}(z))$ is a family of  affine maps that satisfy 
\begin{align*}
\kappa_t(0) & = 0\\
    \lim_{z \mapsto \infty} \frac{\kappa_t(z)}{z} &= -1
\end{align*}
Therefore, $ \kappa_t(z) = -z $.

$\psi_t \circ p_a \circ \psi_t^{-1}$ is a polynomial with exactly two critical points that commutes with $\kappa_t$. It has the form $\widetilde{a}(t){\displaystyle \int_{0}^{z}} \Big(1-\frac{w}{x(t)}\Big)^d\Big(1+\frac{w}{x(t)}\Big)^ddw$. The functions 
$\widetilde{a}, x:\D_{\frac{1}{k}} \longrightarrow \C$ are holomorphic, with $x(0) = \sqrt{d}$, $\widetilde{a}(0)=a$. These polynomials are odd, therefore by conjugating them by $h_t(z) =  \frac{z\sqrt{d}}{x(t)}$, we obtain polynomials of the form $a(t){\displaystyle \int_0^z}\Big(1-\frac{w^2}{d}\Big)^ddw$, where $a: \mathbb{D}_{\frac{1}{k}} \longrightarrow \C$ is holomorphic, with $a(t) \in \mathcal{CBO}_d$, and $a(0) = a \in \partial \mathcal{CBO}_d$. But this implies that $a(t)$ is a constant function, and $\psi \circ \psi_1^{-1} \circ h_1^{-1}$ is a hybrid equivalence between $p_a$ and $p_{a'}$. 
Lemma~\ref{lemm:quasi} implies  $a=a'$.\\
\end{proof}

\noindent To show continuity of $\Phi_d$ at $c \in \partial \mathcal{M}_{d+1}$, it suffices to show that its graph is closed, that is, if $c_n \in \mathcal{M}_{d+1}$ converge to $c$ and $a_n=\Phi_d(c_n) \longrightarrow \widetilde{a}$, then $\widetilde{a} = \Phi_d(c)$.

Let 
\begin{align*}
a&=\Phi_d(c) &f_n & = f_{c_n}\\
g&=g_c & g_n&=g_{c_n}\\
\sigma & = \sigma_c &\sigma_n&=\sigma_{c_n}\\
\psi & = \psi_c & \psi_n & = \psi_{c_n}\\
G&=G_c = \psi \circ g \circ \psi^{-1} &G_n& = G_{c_n} = \psi_n \circ g_n \circ \psi_n^{-1}\\
 \varphi &=\varphi_c   &\varphi_{n} & =\varphi_{c_n}\\
 V_n& = (V_0)_{c_n} & V & = (V_0)_c\\
 Q_n(\ell) & = \Big(Q_1\Big(\frac{\ell}{d+1}\Big)\Big)_{c_n} &  Q(\ell) & = \Big(Q_1\Big(\frac{\ell}{d+1}\Big)\Big) _c\\
 T_n(\ell) & = \Big(T_1\Big(\frac{\ell}{d+1}\Big)\Big)_{c_n} &  T(\ell) & = \Big(T_1\Big(\frac{\ell}{d+1}\Big)\Big) _c
\end{align*}
    \begin{prop}
   The sequence of quasiregular maps $g_n$ converge to $g$.
    \end{prop}
    \begin{proof}
    On both the $+$ and $-$ copies of $(W_0)_{c_n}$, $g_n$ coincides with $f_n$ away from the sectors $Q_1(\ell)$, for $\ell=1,2,...,d$. On each of these sectors, $g_n$ has as its components a conformal map $\widetilde{h}_n: T_n(\ell) \longrightarrow V_n$, chosen uniquely so that the triple $((z_1)_n , (z_2)_n, \omega^\ell\beta_n)$ is mapped to the triple $((\widetilde{z}_1)_n, (\widetilde{z}_2)_n), \beta_n)$, and a quasiconformal extension to $Q_n(\ell) \setminus \overline{T_n(\ell)}$. Similarly, on $T(\ell)$, $g$ agrees with an analytic map $\widetilde{h}: T(\ell) \longrightarrow V$ chosen so that the triple $(z_1,z_2,\omega^\ell\beta)$ is sent to $(\widetilde{z}_1, \widetilde{z}_2, \beta)$. Fix an  $\ell \in \{1,2,...,d\}$. We will first show that the $\widetilde{h}_n$  converge to $\widetilde{h}$.

Let $\rho_n$ be the Riemann map that sends $\D$ to $V_n$, with $\rho_n(0) = 0$ and $\rho_n'(0) >0$. Observe that $V_n$ converges to $V$ with respect to the point $0$ in the sense of kernel convergence (see \cite[Section~1.4]{pommerenke}). By Carath\'{e}odory's kernel convergence theorem (\cite[Theorem~1.8]{pommerenke}), $\rho_n \longrightarrow \rho$ uniformly in $\D$, where $\rho :\D \longrightarrow V$ is a conformal map that sends $0$ to $0$ and satisfies $\rho'(0)>0$. Since the boundaries of $V_n, V$ are quasicircles, the $\rho_n$ extend to $\partial \D$ and these boundary maps converge uniformly to the boundary extension of $\rho$. Thus, the triples $(s_n, t_n, w_n)$ in $\mathbb{S}^1$ that map under $\rho_n$ to $((\widetilde{z}_1)_n, (\widetilde{z}_2)_n), \beta_n)$ converge to the triple $(s,t,w)$ in $\mathbb{S}^1$ that maps under $\rho$ to $(\widetilde{z}_1, \widetilde{z}_2, \beta)$. 

Let $M_n: \D \rightarrow\D $ be a sequence of automorphisms that send $(1,i,-1)$ to $(s_n,t_n, w_n)$, and let $M$ be the automorphism of $\D$ that sends $(1,i,-1)$ to $(s,t,w)$. Then $M_n \longrightarrow M$ on $\overline{\D}$. 

Lastly, for a given $\ell$, note that $\varphi_n(T_n(\ell))$ is the same domain $D = \exp(\Delta_q) = \varphi(T(\ell))$  for each $n$, and furthermore, $(\varphi_n((z_1)_n), \varphi_n((z_2)_n), \varphi_n(\omega^\ell\beta_n)) = (\varphi(z_1), \varphi(z_2), \varphi(\omega^\ell\beta))$ (we note that $\omega^\ell\beta_n$ and $\omega^\ell \beta$ are tips, ie. a unique dynamical ray lands at each of these points, so evaluating the B\"{o}ttcher chart at these points makes sense). Let $e: D \longrightarrow \D$ be the Riemann map that takes the triple $(\varphi(z_1), \varphi(z_2), \varphi(\omega^\ell\beta))$ in $\partial D$ to $(1,i,-1)$. Then 
\begin{align*}
    \widetilde{h}_n & = \rho_n \circ M_n\circ e \circ \varphi_n\\
    \widetilde{h} & = \rho \circ M \circ e \circ \varphi
\end{align*}
It is clear by our discussion that $\widetilde{h}_n \longrightarrow \widetilde{h}$. 

Therefore, on the sectors $T_n(\ell)$, the sequence $g_n$ converges to $g$. But note that the quasiconformal extension to $Q_n(\ell)$ is done in the same way for each $n$. Therefore, $g_n \longrightarrow g$. 

By definition of $\sigma_n$ and $\sigma$, we must have $\sigma_n \longrightarrow \sigma$, and consequently, by the Measurable Riemann Mapping Theorem, $\psi_n\longrightarrow \psi$.

    \end{proof}
\noindent This discussion tells us that 
\begin{align*}
       G_n&= \psi_n \circ g_n \circ \psi_n^{-1} \longrightarrow \psi \circ g  \circ \psi^{-1} =G
    \end{align*}
    Now consider the hybrid equivalences $k_n$ that conjugate $G_n$ to $p_{a_n}$. These have bounded dilitation ratio and map $0$ to $0$, and hence form an equicontinuous family. Upto a subsequence, $k_n$ converge to a quasiconformal map  $\widetilde{k}$. Thus, $k_n \circ G_n \circ k_n^{-1} \longrightarrow \widetilde{k} \circ G \circ \widetilde{k}^{-1}$. We will call the latter map $\widetilde{G}$.
    
    Using \cite[Chapter~II.7, Lemma, p.313]{Douady1985}, $\widetilde{G}$ is quasiconformally  equivalent to $p_{\widetilde{a}}$  (not necessarily hybrid equivalent), but it is also quasiconformally equivalent to $k \circ G \circ k^{-1}$, which in turn is hybrid equivalent to $p_a$. 

This shows that $p_{\widetilde{a}}$ is quasiconformally equivalent to $p_a$. We can choose the equivalence so that the conditions of Lemma~\ref{lemm:quasi} are satisfied. But in order to use this lemma, we also need to show that $a \in \partial \mathcal{CBO}_d$.
    
    Consider a sequence $c_n^*$ of Misiurewicz parameters tending to $c$, and let $a_n^* = \Phi_d(c_n^*)$. Then $a_n^*$ is Misiurewicz, and there exists a subsequence $a_n^* \longrightarrow a^* \in \partial \mathcal{CBO}_d$. By the paragraphs above, $a^* = a$, and hence, $a \in \partial \mathcal{CBO}_d$. Now we apply Lemma~\ref{lemm:quasi} again to get $a = a'$.

\section{Injectivity of $\Phi_d$}\label{section:injectivity}

In this section will construct an inverse $\Psi_d:\mathcal{CBO}_d^{(+,-)} \longrightarrow \mathcal{M}_{d+1}$ of $\Phi_d$.
\subsection{Dynamics of maps in $\mathcal{CBO}_d^{(+,-)}$}\label{section:imgdynamics}
Given $a \in \mathcal{CBO}_d^{(+,-)}$, let $P_{s(a)}$ be the monic representative of $p_a$ as defined in Section~$\ref{section:prelim}$. As in the construction of $\Phi_d$, for $\theta=0,\frac{1}{2}$, let $S_\theta$ be invariant sectors at $0$ with same slope. That is,
\begin{align*}
    S_0 & =\{\varphi_{s(a)}^{-1}(e^{s+2\pi i t}): s \in (0,\eta), |t| < qs\} \\
    S_{\frac{1}{2}} & = \{\varphi_{s(a)}^{-1}(-e^{s+2\pi i t}): s \in (0,\eta), |t| < qs\} 
\end{align*}
Note that $S_{\frac{1}{2}} = -S_0$. 

We choose $q$ to be small enough so that $S_0 \cap S_{\frac{1}{2}} = \{0\}$, and the inverse image of each $S_\theta$ under $P_{s(a)}$ consists of exactly $2d+1$ components.
The point $0$ has pre-images $\{0=x_0, x_1,x_2,...,x_{2d}\}$ under $P_{s(a)}$, of which $d$ - say $x_1,x_2,...,x_d$, are in $F^L_{s(a)}$, and $d$ are in $F^R_{s(a)}$.  Let $S_\theta(x_\ell)$ be the inverse image of $S_\theta$ based at $x_\ell$ for $\ell \neq 0$.

Let $W$ be the region bounded by an equipotential $\{z| G_s(z)=\eta\}$ and define $W_i = P_{s(a)}^{-\circ i}(W)$. For a given $\ell \in \{1,2,...,2d\}$, let $S$ be the connected component of $W_1 \setminus (S_0(x_\ell) \cup S_{\frac{1}{2}}(x_\ell))$ that does not contain $0$. Then $P_{s(a)}$ maps $S$ to $F^L_{s(a)}$ if $S\subset F^R_{s(a)}$, and to $F^R_{s(a)}$ if $S \subset F^L_{s(a)}$. 

We have illustrated this in Figure~\ref{fig:inversedefn}.

\subsection{Definition of $\Psi_d$}\label{section:inversedefn}
With $a$ as above,
\begin{figure}
    \centering
    \includegraphics[scale=0.3]{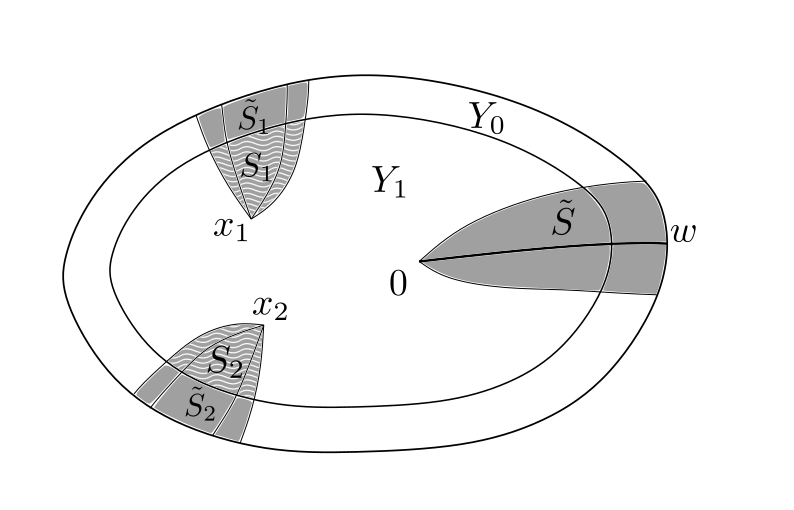}
    \caption{Cut and paste surgery on $P_{s(a)}$}
    \label{fig:odd_to_uni_quasiregular}
\end{figure}
construct the Riemann surface $Y$ as follows: let $Y_0 = W \cap F^L_{s(a)}$, and identify the boundaries $Y_0\cap \mathcal{R}_0(s(a))$ and $Y_0 \cap \mathcal{R}_0(s(a))$ by identifying points on either ray with same speed of escape. Additionally, if necessary, smoothe the boundary of $Y_0$ at the point $w$ as shown in Figure~\ref{fig:odd_to_uni_quasiregular}. $S_0 \cap F^L_{s(a)}$ and $S_{\frac{1}{2}} \cap F^L_{s(a)}$ with this boundary identification become a single sector which we shall call $\widetilde{S}$. We let $Y_1 = P_{s(a)}^{-1}(Y_0)$ with this boundary identification. Clearly, $\overline{Y_1}\subset Y_0$.

Given $\ell \in \{1,2,...,d\}$, let $S$ be the connected component of $Y_1 \setminus (S_0(x_\ell) \cup S_{\frac{1}{2}}(x_\ell))$ that does not contain $0$, and let $S'$ be the component that does. Let $S_\ell = S \cup S_0(x_\ell) \cup S_{\frac{1}{2}}(x_\ell)$ (see Figure~\ref{fig:odd_to_uni_quasiregular}).  Pick a quasiconformal homeomorphism $e_\ell: S_\ell  \mapsto \widetilde{S}$ that extends to a homeomorphism from $\partial S_\ell$ to $\partial \widetilde{S}$, and  coincides with $P_{s(a)}$ on  $\partial S_\ell \cap \partial S'$. For example, this can be constructed in a manner similar to  $g_c \Big| _{\pi_c\big(S_1^\pm\big(\frac{\ell}{d+1}\big)\big)}$ in Section~\ref{section:defn}.
Define
\begin{align*}
    F: Y_1 &\longrightarrow Y_0\\
    F(z) & = \begin{cases}
    P_{s(a)}(z)  & z \in Y_1 \setminus \bigcup_{\ell=1}^d  S_\ell \\
    e_\ell(z) & z \in S_\ell \text{ for some }\ell \in \{1,2,...,d\}
    \end{cases}
\end{align*}
$F$ is clearly a quasiregular map of degree $d+1$ with a single critical point. We define an $F-$ invariant complex structure $\sigma$ on $Y_0$ as\\
\begin{itemize}
    \item $E_z=\mathbb{S}^1$ if $z \in Y_0 \setminus Y_1$ or the $F-$ orbit of $z$ does not intersect $S_\ell$ for any $\theta,\ell$ \\
    \item $E_z = (DF^{n})^{-1}(\mathbb{S}^1)$ if $F^{n}(z)$ is the first point in the orbit of $z$ that is in $S_\ell$\\
\end{itemize}
\begin{figure}
    \centering
    \includegraphics[scale=0.4]{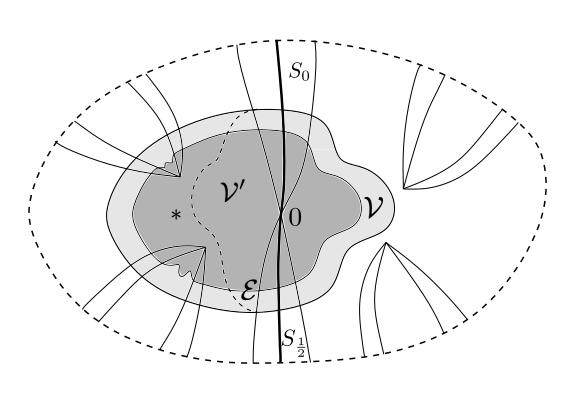}
    \caption{Alternative construction of $\Psi_d(a)$ by  renormalization; the 
    `$*$' marks the critical point $-\sqrt{d}s(a)$ of $P_{s(a)}$}
    \label{fig:odd_to_uni_renorm}
\end{figure}
Every $F-$ orbit visits $S_\theta(x_i)$ at most once. So, $\sigma$ has bounded dilitation. Note that $F^*\sigma = \sigma$, and thus, $F$ is quasiconformally equivalent to a polynomial-like map $y: V_1 \longrightarrow V$ with degree $d+1$ and a single critical point. The map $y$ is hybrid equivalent to a polynomial of the form $f_c(z) = z^{d+1}+c$. Note that $y$ only determines the affine  equivalence class of $f_c$, and thus $c$ is not unique if $d>1$; however, we impose the condition that the identified rays $\mathcal{R}_0(s(a))$ and $\mathcal{R}_{\frac{1}{2}}(s(a))$ are eventually mapped to the same access as the dynamical ray at angle $0$ to $f_c$ (with respect to the B\"{o}ttcher chart where  $\frac{\varphi_c(z)}{z} \rightarrow 1$ as $z \rightarrow \infty$). This determines $c$ uniquely. It is clear that $c \in \mathcal{M}_{d+1}$; we therefore define $\Psi_d(a) = c$. 
\begin{rem}
We may also construct $\Psi_d(a)$ by choosing a renormalization of $P_{s(a)}$.

Choose a neighborhood $\mathcal{E}$ of $0$, in which $P_{s(a)}$ is conjugate to $z \mapsto rz$ for some $r \in \C$ with $|r|>1$, small enough so that $\overline{\mathcal{E}}$ does not contain any critical points, and satisfying 
\begin{align*}
    \mathcal{E} \cap S_0 &= S_0 \cap W_i\\
    \mathcal{E} \cap S_{\frac{1}{2}} &= S_\frac{1}{2} \cap W_i
\end{align*}

Let $\mathcal{V}$ be an open set defined the union of $W_i \cap F^L_{s(a)}$ and $\mathcal{E}$. Then, there exists a connected component $\mathcal{V}'$ of $P_{s(a)}^{-1}(\mathcal{V})$ such that $\overline{\mathcal{V}'} \subset \mathcal{V}$, and $P_{s(a)}\big |_{\mathcal{V}'} : \mathcal{V}' \longrightarrow \mathcal{V}$ is polynomial-like of degree $d+1$ (see Figure~\ref{fig:odd_to_uni_renorm}). This polynomial-like map has a unique critical point at $-\sqrt{d}$, and by the straightening theorem, it is hybrid equivalent to a unicritical degree $d+1$ polynomial.

We can show for an appropriate choice of domains, the map  $F$ defined above in the first definition of $\Psi_d(a)$ and $P_{s(a)} \big |_{\mathcal{V}'}$ are hybrid equivalent.
\end{rem}

We may use the same methods as in Section~\ref{section:continuity} to show that $\Psi_d$ is continuous.
\subsection{$\Psi_d$ is the inverse of $\Phi_d$}
Given $c \in \mathcal{M}_{d+1}$, let ${c'} = \Psi_d \circ \Phi_d(c)$. We will follow the construction to show that $f_{c'}$ and $f_c$ are hybrid equivalent, and thus, $c'=c$. 

Let $a=\Phi_d(c)$. The construction $a \mapsto \Psi_d(a)$ involves picking the sectors $S_0$ and $S_\frac{1}{2}$ in the dynamical plane of $P_{s(a)}$, constructing a Riemann surface $ Y$,  a quasiregular map $F_{s(a)}$, and lastly, a polynomial like map $y_{s(a)}$. 

On the other hand, the construction $c \mapsto \Phi_d(c)$ goes through the steps $f_c \mapsto g_c \mapsto G_c \mapsto P_{s(a)}$. We will only be working with the `$-$' copies of  $S(\frac{\ell}{d+1}), K_{f_c}$, etc., and so we shall drop the `$-$' superscript. The first step in the construction of $\Phi_d(c)$ uses the quotient map $\pi_c$, and we have
\begin{align*}
    g_c (\pi_c(z))&= \pi_c(f_c(z)) \text{ away from sectors }\pi_c\Big(S\Big(\frac{\ell}{d+1}\Big)\Big)\\
    G_c& = \psi_c \circ g_c \circ \psi_c^{-1}\\
    P_{s(a)} & = k_c \circ G_c \circ k_c^{-1}
\end{align*}
where $\psi_c$ is quasiconformal and $k_c$ is a hybrid equivalence.
   \begin{figure}
    \centering
    \includegraphics[scale=0.4]{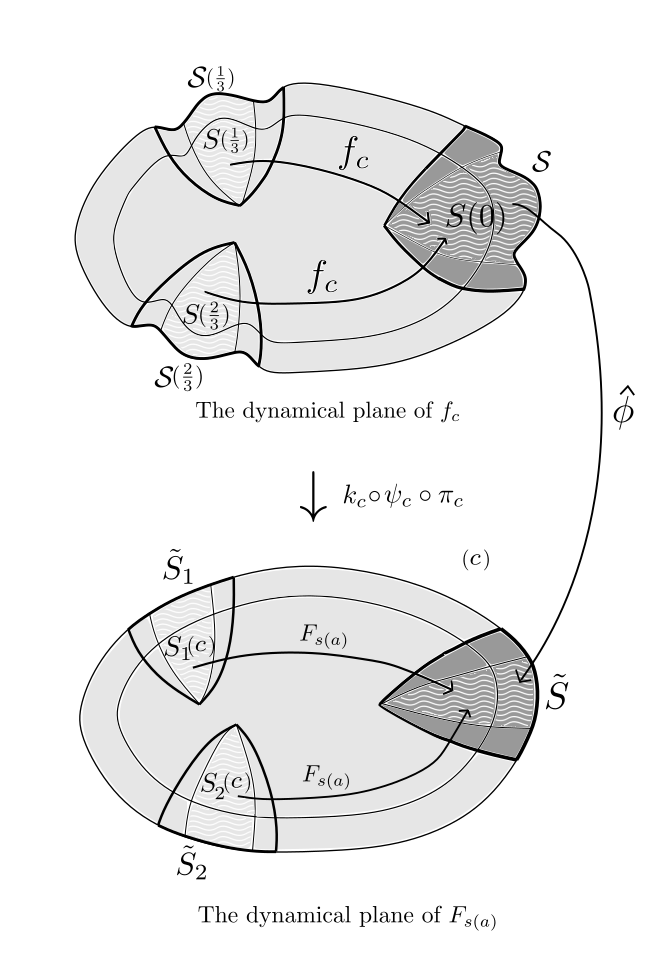}
    \caption{Building a conjugacy between $f_c$ and $F_{s(a)}$. The wavily shaded region in the top figure is $S(0)$; it is contained in $\mathcal{S}$ and its two copies map under $k_c \circ \psi_c \circ \pi_c$ to the sectors $S_0$ and $S_1$ respectively, which we cut to make $\tilde{S}$.  We define $\hat{\phi}$ on the darkly shaded region on the top to the darkly shaded region at the bottom.}
    \label{fig:inverse}
\end{figure}
\\In the dynamical plane of $P_{s(a)}$, for $\ell \in \{1,2,...,d\}$, define
\begin{align*}
    \widetilde{S}_0(x_\ell) & =\Big\{\varphi_{s(a)}^{-1}(e^{r+2\pi i t}): s \in (0,\eta), \Big|t-\frac{\ell}{2d}\Big| < qs\Big\} \\
    \widetilde{S}_{\frac{1}{2}}(x_\ell) & =\Big\{\varphi_{s(a)}^{-1}(-e^{s+2\pi i t }): s \in (0,\eta), \Big|t-\frac{\ell}{2d}\Big| < qs\Big\} 
\end{align*}
and let $\widetilde{S}_\ell$ be the union of $\widetilde{S}_0(x_\ell), \tilde{S}_{\frac{1}{2}}(x_\ell)$ and the connected component of $Y_0 \setminus $
$\widetilde{S}_0(x_\ell) \cup \tilde{S}_{\frac{1}{2}}(x_\ell)$ that contains $S_\ell$, as defined in Section~\ref{section:inversedefn}. See Figure~\ref{fig:odd_to_uni_quasiregular} for an illustration of $\widetilde{S}_\ell$. Let $\widetilde{\phi} = k_c \circ \psi_c \circ \pi_c$. 

In the dynamical plane of $f_c$, let $S_1\big(\frac{\ell}{d+1}\big)$ , $\ell = 0,1,...,d$, be as defined in Equations~\ref{eqn:s1defn} to \ref{eqn:y1defn} (the equipotential $\eta$ and the slope factor $q$ may be different from the ones used for $P_{s(a)}$). There are two copies of $S(0) = f_c\big(S_1(0)\big)$ in the dynamical plane of $G_c$, but we will pick the copy that eventually gets mapped to a sector that intersects  $S_0$. More generally, for a suitable choice of equipotential and slope factor in the $f_c$ - plane, we may assume that the open sets $S\big( \frac{\ell}{d+1}\big) = \omega^\ell S(0)$ are eventually mapped into $\widetilde{S}_\ell$, and that $S(0)$ is eventually mapped to $S_0$ (or to $S_{\frac{1}{2}}$). That is,
\begin{align*}
S_\ell(c) &=\widetilde{\phi} \Big(S\Big( \frac{\ell}{d+1}\Big)\Big) \subset \widetilde{S}_\ell  \text{ for } \ell = 0,1,...,d\\
\tilde{\phi}(V_0) &= Y_0 \setminus S_0(c) 
\end{align*}
where the domain $V_0$ is as defined in Equation~\ref{eqn:vodefn}.

  Our strategy will be to set up a quasiconformal map $\phi: V_0 \cup S(0) \longrightarrow Y_0$ that has agrees with $\widetilde{\phi}$ away from certain sectors, and conjugates $f_c$ and $F_{s(a)}$. \\
Let 
\begin{align*}
    V& = \widetilde{\phi}^{-1}(Y_0 \setminus \widetilde{S})\\
    V_1 & = f_c^{-1}(V)\\
    \mathcal{S} & = V_0 \cup S(0) \setminus V\\
    \mathcal{S}\Big(\frac{\ell}{d+1}\Big) & = \widetilde{\phi}^{-1}(\tilde{S}_\ell) \text{ for }\ell = 1,2,...,d
\end{align*}
\noindent See Figure~\ref{fig:inverse} for details. For all $z \in V_1 \setminus \Big(\mathcal{S} \cup \bigcup_\ell \mathcal{S}\big(\frac{\ell}{d+1}\big)\Big)$,
\begin{align*}
    F_{s(a)} \circ \widetilde{\phi}(z) &= \widetilde{\phi} \circ f_c(z) 
\end{align*}
Furthermore, with degree one,
\begin{align*}
    f_c\Big(\mathcal{S}\Big(\frac{\ell}{d+1}\Big) \cap V_1\Big) &= \mathcal{S} \text{ for }\ell = 1,2,...,d\\
    f_c(\mathcal{S} \cap V_1) &=\mathcal{S}
\end{align*}
For $z \in \mathcal{S}$, $z = f_c(w)$ for $d$ distinct $w \in V_1$. We can assume that $F_{s(a)} \circ \widetilde{\phi}(w)$ does not depend on the choice of preimage $w$, since $\widetilde{\phi}(w) \in S_\ell$ and $F_{s(a)}\big|_{S_\ell}$ depends on the homeomorphisms $e_\ell$ defined as in Section~\ref{section:inversedefn}, which we have freedom in choosing.\\
So we set 
\begin{align*}
    \hat{\phi}(z) & = F_{s(a)} \circ \widetilde{\phi}(w)
\end{align*}
Define 
\begin{align*}
    \phi&: V_0 \cup S(0) \longrightarrow Y_0\\
    \phi(z) & = \begin{cases}
    \widetilde{\phi}(z) & z \not \in \mathcal{S}\\
    \hat{\phi}(z) & z \in \mathcal{S}
    \end{cases}
\end{align*}
By the discussion above, for all $z \in f_c^{-1}(V_0 \cup S(0))$, 
\begin{align*}
    F_{s(a)} \circ \phi(z) & = \phi \circ f_c(z)
\end{align*}

We note that $\pi_c$ changes the angle at $\beta_c$ from $2\pi$ to $\pi$, and has zero dilitation on $K_{f_c} \setminus \{\beta_c\}$.  Also note that $\psi_c$ has zero dilitation on $\pi_c(K_{f_c})$. 

On the other hand, the cutting procedure in Section~\ref{section:inversedefn} changes the angle $\pi$ made  by the boundary of $F^L_{s(a)}$ at $0$ to the angle $2 \pi$ in the plane of $F_{s(a)}$. Lastly, note that $k_c$ has zero dilitation on $\psi_c\circ \pi_c(K_{f_c}) \setminus \{0\}$.

Combined, this information tells us that we have constructed a quasiconformal map $\phi: V_0 \cup S(0) \longrightarrow Y$
that has zero dilitation on $K_c$, and conjugates $f_c$ to $F_{s(a)}$. 

Now, if $z \in \phi(K_c)$, a point  $F^{\circ n}_{s(a)}$ in the orbit of $z$ cannot be in the interior of $S_\ell$ for any $\ell$. Therefore, the quasiconformal map that conjugates $F_{s(a)}$ to $y_{s(a)}$ has zero dilitation on $\phi(K_c)$. That is, $y_{s(a)}$ and $f_c$ are hybrid equivalent. Thus, $f_{c'}$ and $f_c$ are hybrid equivalent,  implying $c=c'$. 

In a similar manner, we can show that  $p_{\widetilde{a}}$, where $\widetilde{a} = \Phi_d \circ \Psi_d(a)$, is hybrid equivalent to $p_a$. That is, $\Phi_d \circ \Psi_d(a) = a$.

This finishes the proof of Theorem~\ref{thm:maintheorem}.

We will end with a discussion of how the image under $\Phi_d$ fits inside $\mathcal{CBO}_d$.
\begin{lemma}
$\mathcal{CBO}^{(+,-)}_d$ disconnects $\mathcal{CBO}_d$ into infinitely many components. 
\end{lemma}
\begin{proof}
Let $f_c(z) = z^{d+1}+c$ be a polynomial where the orbit of $c$  contains the $\beta-$fixed point where the dynamical ray at angle $0$ lands. There are infinitely many values of $c$ in $\mathcal{M}_{d+1}$ that satisfy this condition - these are precisely the landing points of parameter rays at angles $\frac{i}{d^n}$ for $n\geq 1$ and $0 < i < d^n$. These are included in the set of ``tips'' of $\mathcal{M}_{d+1}$. 

Given such a $c$, let $a=\Phi_d(c)$. Then the orbit of both critical points $\pm \sqrt{d}$ of $p_a$ eventually lands on $0$ - that is, there exists $k$ such that $p_a^{\circ k}(\pm \sqrt{d}) = 0$.

In the dynamical plane of the monic representative $P_{s(a)}$, the dynamical rays at angles $0,\frac{1}{2}$ land at $0$. Thus there exist two angles $\theta_1,\theta_2$ such that $(2d+1)^{k-1} \theta_1  \equiv 0 $ and $(2d+1)^{k-1} \theta_2 \equiv \frac{1}{2}$, which both land at the critical value $P_{s(a)}\big(-s(a)\sqrt{d}\big)$. In the parameter plane of $\mathcal{MBO}_d$, the parameter rays at angle $\theta_1,\theta_2$ both land at $s(a)$- which means that $s(a)$ is a cut-point of $\mathcal{MBO}_d$, which is equivalent to saying that $a$ is a cut-point of $\mathcal{CBO}_d$.

Another way to show this is to see that  exists $a' \in \mathcal{CBO}_d$ close to $a$ such that  $P_{s(a')}^{
\circ k}(\sqrt{d}) \in F^L_{s(a')}$ and $P_{s(a')}^{
\circ k}(-\sqrt{d}) \in F^R_{s(a')}$. That is, the orbits of both critical points eventually ``cross over'' to the other side. So $a' \not \in \mathcal{CBO}_d^{(+,-)}$.
\end{proof}
\bibliography{bibtemplate}
\bibliographystyle{amsalpha}
\end{document}